\documentclass[10.5 pt]{amsart}
\usepackage{amsmath,amsthm,amssymb,graphicx,color}
\usepackage{epsfig,amsfonts,latexsym}
\usepackage{enumerate}
\allowdisplaybreaks
\usepackage{bm}
\newcommand*{\B}[1]{\ifmmode\bm{#1}\else\textbf{#1}\fi}
\usepackage[round]{natbib}
\usepackage{hyperref}
\hypersetup{colorlinks,
citecolor=blue,
linkcolor=blue,
urlcolor=black
}

\numberwithin{equation}{section}
\newcommand{\eqdef}{\stackrel{d}{=}}
\newcommand{\eqdistn}{\stackrel{\mathcal{L}}{=}}

\newcommand{\bX}{\mathbf{X}}
\newcommand{\bY}{\mathbf{Y}}
\newcommand{\bW}{\mathbf{W}}
\newcommand{\bZ}{\mathbf{Z}}

\newcommand{\Zd}{\mathbb{Z}^d}

\newcommand{\SaS}{\text{S}\alpha \text{S}}

\newcommand{\sas}{S \alpha S}
\newtheorem{thm}{Theorem}[section]
\newtheorem{remark}[thm]{Remark}
\newtheorem{propn}[thm]{Proposition}

\newtheorem{defn}[thm]{Definition}

\newtheorem{cor}[thm]{Corollary}

\def\bbR{\mathbb R}
\def\bbZ{\mathbb Z}

\def\P{\mathbb P}
\def\E{\mathbb E}

\theoremstyle{definition}

\newcount\Comments  
\definecolor{darkgreen}{rgb}{0,0.5,0}
\definecolor{purple}{rgb}{0.3,0.1,0.3}
\newcommand{\kibitz}[2]{\ifnum\Comments=1\textcolor{#1}{#2}\fi}

\allowdisplaybreaks

\begin{document}
%



   \author{Snigdha Panigrahi}
   \address{Department of Statistics, University of Michigan, Ann Arbor, MI 48109, USA}
   \email{psnigdha@umich.edu}


   \author{Parthanil Roy}

   \address{Theoretical Statistics and Mathematics Unit, Indian Statistical Institute, Bangalore 560059, India}


   \email{parthanil.roy@gmail.com}


   \author{Yimin Xiao }

   \address{Department of Statistics and Probability, Michigan State University, East Lansing, MI 48824-1027, USA}


   \email{xiao@stt.msu.edu}


   \title{Maximal Moments and Uniform Modulus of Continuity for Stable Random Fields}


\begin{abstract}
In this work, we solve an open problem mentioned in \cite{xiao:2010} and provide sharp bounds on the rate of growth of maximal moments for stationary symmetric stable random fields using structure theorem of finitely generated abelian groups and ergodic theory of quasi-invariant group actions. We also investigate the relationship between this rate of growth and the path regularity properties of self-similar stable random fields with stationary increments, and establish uniform modulus of continuity of such fields. In the process, a new notion of weak effective dimension is introduced for stable random fields and is connected to maximal moments and path properties. Our results establish a boundary between shorter and longer memory in relation to H\"older continuity of S$\alpha$S random fields confirming a conjecture of \cite{samorodnitsky:2004a}.
\end{abstract}



 \thanks{The research of P. Roy was partially supported by the project RARE-318984
(a Marie Curie FP7 IRSES Fellowship), a SERB grant MTR/2017/000513 and Cumulative Professional Development Allowance at Indian Statistical Institute. The research of Y. Xiao was partially supported by NSF
grants DMS-1612885 and DMS-1607089. }

 \dedicatory{Dedicated to the memory of Professor Wenbo Li}



 \maketitle



\section{Introduction}

A real-valued stochastic process $\{X(t): t \in \mathbb{T}^d\}$ ($\mathbb{T} = \mathbb{Z}$ or $[0,1]$ or $\mathbb{R}$) is
called a \emph{symmetric $\alpha$-stable} ($\sas$) random field if each of its finite linear combination follows an $\sas$ distribution. In general, the parameter $\alpha$ satisfies $0 < \alpha \leq 2$, although in this paper, we
assume our random fields to be non-Gaussian and therefore $0 < \alpha < 2$. See, for example,
\cite{samorodnitsky:taqqu:1994} for detailed discussions on non-Gaussian stable distributions and processes.

Sample path continuity and H\"older regularity of stochastic processes and random fields have been studied
for many years. The main tool behind such investigation has been a powerful chaining argument that is mainly
applicable to Gaussian and other light-tailed processes; see  \cite{Adler_Taylor_2007},
\cite{khoshnevisan2002multiparameter},  \cite{marcus2006markov},
\cite{talagrand2006generic}. Recently, there has been a significant interest in establishing uniform modulus
of continuity of sample paths for stable and other non-Gaussian infinitely divisible processes; see, for instance,
\cite{ayache:roueff:xiao:2009}, \cite{bierme2009holder, bierme2015modulus}, \cite{xiao:2010}.

Motivated by \cite{kono:maejima:1991a}, \cite{xiao:2010} modified the existing chaining argument and
made it amenable to heavy-tailed random fields. This technique uses estimates of the lower order moments
of the maximum increments over the two consecutive steps of the chain to obtain a uniform modulus of continuity
for stable and other heavy-tailed random fields.

In this context, it was stated in \cite{xiao:2010} (see Page 173 therein) that for a
stationary $\alpha$-stable sequence $\{\xi_k: k \geq 1\}$, it is an open problem to give sharp upper and lower
bounds for the maximal moment sequence $\E\big(\max_{1 \leq k \leq n} |\xi_k|^\gamma\big)$ for $\gamma \in (0, \alpha)$. \cite{xiao:2010} also presented two approaches of partial solution to this open problem: one using results of \cite{talagrand2006generic} in this setup and another one based on his own improvement of Talagrand's results (more specificallty, Lemma~3.5 of \cite{xiao:2010}). However both of these methods lead to weaker path continuity results and we have been able to improve them significantly in this paper as described below.

We have solved the aforementioned open problem (see Theorems \ref{thm1:maximal_moment}, \ref{weak:eff:dim} and \ref{thm3:maximal_moment} below) of deriving sharp bounds on the moments of the maximal process for stationary $\sas$ discrete-parameter random fields having various generic dependence structures based on ergodic and group theoretic properties of the underlying \emph{nonsingular} (also known as \emph{quasi-invariant}) group action. The maximum is taken over usual $d$-dimensional hypercubes of side-length increasing to infinity. Our machineries include structure theorem for finitely generated abelian groups, ergodic theory of quasi-invariant actions on $\sigma$-finite standard measure spaces, and a new notion of weak effective dimension introduced in this paper. This work easily extends to the continuous parameter case (see Theorem~\ref{cont:parameter:gom} and Remark~\ref{remark:last}) provided the random field is measurable and stationary.

Solution to the open problem in the discrete-parameter case allows us to prove results on uniform modulus
of continuity for a large class of self-similar $\sas$ random fields with stationary increments; see Section~\ref{sec:path:properties}. To this end, we have introduced a novel notion, namely that of \emph{weak effective dimension}, for stable random fields; see Definition~\ref{weak:eff:dim:defn} below. This notion encompasses the concept of effective dimension (defined by \cite{roy:samorodnitsky:2008}) as a special case and connects naturally to maximal moments (see Theorem~\ref{weak:eff:dim}) and path properties (see Corollary~\ref{label-result in case of weak effective dimension}) of stable random fields. In some sense, our new notion is better than the effective dimension, which is always an integer (and hence more restrictive) whereas weak effective dimension need not be so.

Based on the ergodic theoretic properties of the underlying nonsingular group action, \cite{samorodnitsky:2004a}
obtained a phase transition boundary for the partial maxima of stable processes. It was also conjectured in this work that many other important phase transitions for stable processes should occur at this boundary. While this conjecture has been established for  ruin probabilities (see \cite{mikosch:samorodnitsky:2000a}), growth of maxima (see \cite{roy:samorodnitsky:2008}, \cite{roy:2010b}), extremal point processes (see \cite{resnick:samorodnitsky:2004}, \cite{roy:2010a}), large deviations issues (see \cite{fasen:roy:2016}), statistical aspects (see \cite{bhattacharya:roy:2018}), etc., the effects of this transition boundary on path properties have not yet been explored.

In this work, we bridge this gap and establish that the uniform modulus of continuity does change significantly at this boundary (see Section~\ref{sec:path:properties}) and also at the group theoretic boundaries obtained by \cite{roy:samorodnitsky:2008}. This confirms the aforementioned conjecture of \cite{samorodnitsky:2004a} (see Page 1440 therein and also the discussions in the beginning of Page~174 in \cite{xiao:2010}) in the context of path behaviours of stable random fields. More specifically, the confirmation of this conjecture is attained through Corollaries \ref{label-result in case of dissipative flow}, \ref{label-result in case of effective dimension} and \ref{label-result in case of weak effective dimension}, all of which follow from Theorem~\ref{label-result in case of tightness} below.

We would also like to mention that our bounds (on both growth-rate of maximal moments as well as uniform modulus of continuity) are significantly better than the existing ones for stable random fields that are not full-dimensional (see Definition~\ref{weak:eff:dim:defn} below) to the extent that we improve the leading (polynomial) term of these bounds. On the other hand, in the full-dimensional case (mixed moving average, for example) the improvement is in the logarithmic term albeit nontrivial.

This paper is organized as follows. In Section~\ref{sec:prelim}, we recall a result of \cite{xiao:2010},
explain how it naturally leads to a problem on rate of growth of the maximal moment sequences and describe the ergodic theoretic and group theoretic connections to this extreme value theoretic problem. Subsection~\ref{sec:contri} contains a brief summary of the contributions of our work. In Section~\ref{sec:discrete}, we state the results on the
asymptotic behavior of the maximal moments of stationary
$\sas$ random fields as the index parameter runs over $d$-dimensional hypercubes of
increasing edge-length even though the proofs are deferred to Section~\ref{sec:proofs} to increase the readability of this paper. In Section~\ref{sec:path:properties}, we establish results on uniform modulus
of continuity for self-similar $\sas$ random fields whose first order increments are stationary. Two important
examples of fractional stable processes are discussed in Section~\ref{sec:examples}. Finally in Appendix~\ref{Appendix}, we present a result on the growth-rate of maximal moments in the continuous parameter case.

Throughout this paper,  we will use $K$ to denote a  positive and finite
constant which may differ in each occurrence, even in two consecutive ones. Some specific constants
will be denoted by $c, c_1, c_2, \ldots, K_{1}, K_2, \ldots$, etc. For two sequences of nonzero real numbers $\{a_n\}_{n\in\mathbb{N}}$ and $\{b_n\}_{n\in\mathbb{N}}$ the
notation $a_n \sim b_n$ means $a_n/b_n \to 1$ as $n \to \infty$. For $u, v \in \mathbb{R}^d$, $u = (u_{1}, u_{2}, \ldots, u_{d}) \leq v =(v_{1},
v_{2}, \ldots, v_{d})$ means $u_{i} \leq v_{i}$ for all $i=1, 2, \ldots , d$. The vectors
$\mathbf{0}=(0,0,\ldots,0)$, $\mathbf{1}=(1,1,\ldots,1)$ are elements of $\mathbb{Z}^d$. 
We shall abuse the notation and use $[u, v]$ to denote the set $\{t \in \mathbb{Z}^d: u \leq t \leq v\}$ or the set
$\{t \in \mathbb{R}^d: u \leq t \leq v\}$ depending on the context (the former notation is used throughout the main body of the paper while latter one is used only in the appendix). For $\alpha \in (0, 2)$ and a $\sigma$-finite standard measure space $(S,\mathcal{S}, \mu)$, we define the space $L^{\alpha}(S,\mu):=\left\{f:S\to\mathbb{R} \mbox{ measurable}:
\|f\|_\alpha <\infty \right\}$, where $\|f\|_\alpha:=\left(\int_S|f(s)|^{\alpha}\,\mu(ds)\right)^{1/\alpha}$. Note that $\|\cdot\|_\alpha$ is a norm if and only if $\alpha \in [1,2)$ making the corresponding $L^{\alpha}(S,\mu)$ a Banach space but not a Hilbert space.
For two random variables $Y$, $Z$, we write $Y\eqdistn Z$ if $Y$ and $Z$ are identically distributed.
For two stochastic processes $\{Y(t)\}_{t \in T}$ and $\{Z(t)\}_{t \in T}$, the notation $\{Y(t)\}  \eqdistn \{Z(t)\}$
(or simply $Y(t) \eqdistn Z(t)$, $t \in T$) means that they have the same finite-dimensional distributions.


\section{Preliminaries} \label{sec:prelim}
\subsection{A chaining argument for path properties}
We start with a brief description of the main result in \cite{xiao:2010}, which is built upon a modification
of the chaining arguments used in the proofs of Kolmogorov's continuity theorem, Dudley's entropy theorem
and other results on path regularity properties in the light-tailed situations; see \cite{Adler_Taylor_2007},
\cite{khoshnevisan2002multiparameter},  \cite{marcus2006markov}, \cite{talagrand2006generic}. To this end,
let  $\{X(t)\}_{t \in T}$ be a random field indexed by a compact metric space $(T, \rho)$, and let
$\{D_n: n\geq 1\}$ be a sequence (which is called a chaining sequence) of finite subsets
of $T$ satisfying the following conditions:
\begin{enumerate}
\item \label{property_finitely_many_nbrs}
There exists a positive integer $\kappa_0$ depending only on $(T,\rho)$ such that for every $\tau_n \in D_n$, the set
    \[
    O_{n-1}(\tau_n):=\{\tau_{n-1}^\prime \in D_{n-1}: \rho(\tau_n, \tau_{n-1}^\prime) \leq 2^{-n}\}
    \]
    has at most $\kappa_0$ many elements.

\item \label{property_chaining}  (The Chaining Property)
For every $s, t \in T$ with $\rho(s,t) \leq 2^{-n}$, there exist
two sequences $\{\tau_p(s): p \geq n\}$ and $\{\tau_p(t): p \geq n\}$ such that $\tau_n(s) = \tau_n(t)$ and, for every
$p \geq n$, $\tau_p(s), \tau_p(t) \in D_p$, $\rho(\tau_p(s),s) \leq 2^{-p}$,  $\rho(\tau_p(t),t) \leq 2^{-p}$, and
$\tau_p(s) \in O_p(\tau_{p+1}(s))$,  $\tau_p(t) \in O_p(\tau_{p+1}(t))$. If $s \in D:=\cup_{k=1}^\infty D_k$ (if $t \in D$),
then there exists an integer $q \geq 1$ such that $\tau_p(s)=s$ ($\tau_p(t)=t$, resp.) for all $p \geq q$.
\end{enumerate}


Note that Condition~\eqref{property_chaining} yields immediately that for each $n$, $T$ can be covered by open
balls with radius $2^{-n}$ and centers in $D_n$, and the set $\cup_{n \geq 1} D_n$ is dense in $T$. The following
result from \cite{xiao:2010} provides an upper bound for the uniform
modulus of continuity for a class of random fields, including those with heavy-tailed distributions.

\begin{propn}
\label{propn:xiao_2010}
Let $\mathbf{X}=\{X(t)\}_{t \in T}$ be a real-valued random field indexed by a compact metric space $(T, \rho)$ and
let $\{D_n: n \geq 1\}$ be a chaining sequence satisfying Conditions \eqref{property_finitely_many_nbrs} and
\eqref{property_chaining} above. Suppose $\sigma: {\mathbb R}_+ \to {\mathbb R}_+$ is a nondecreasing
continuous function which is regularly varying at the origin with index $\delta > 0$ (i.e., $\lim_{h \to 0+}\sigma(ch)/\sigma(h)
= c^\delta$ for all $c>0$). If there are constants $\gamma > 0$, and $K>0$ such that
\begin{equation}
\E \bigg(\max_{\tau_n \in D_n}\,\max_{\tau^\prime_{n-1} \in O_{n-1}(\tau_n)}|X(\tau_n)-X(\tau^\prime_{n-1})|^\gamma\bigg)
\leq K\, \big(\sigma(2^{- n})\big)^\gamma
\end{equation}
for all integers $n \geq 1$, then for all $\epsilon >0$,
\begin{equation}
\lim_{h \to 0+} \frac{\sup_{t \in T}\sup_{\rho(s,t)\leq h}|X(t)-X(s)|}{ \sigma(h) (\log{1/h})^{(1+\epsilon)/\gamma}} =0
\end{equation}
almost surely.
\end{propn}

In this paper, we will focus on studying the maximal moments of $\sas$ random fields indexed by $\mathbb Z^d$ or
$\mathbb R^d$ so that we can apply Proposition \ref {propn:xiao_2010} to self-similar $\sas$ random fields with
stationary increments. Recall that a  random
field $\{X(t)\}_{t \in \mathbb{R}^d}$ is called $H$-self-similar ($H>0$) if $\{X(ct)\}_{t \in \mathbb{R}^d}\eqdistn
 \{c^H X(t)\}_{t \in \mathbb{R}^d}$
for all $c>0$. 
$\{X(t)\}_{t \in \mathbb{R}^d}$ is said to have stationary increments if, $\{X(t+u) - X(u)\}_{t \in \mathbb{R}^d}
\eqdistn \{X(t) - X(0)\}_{t \in \mathbb{R}^d}$,  for each $u \in \mathbb{R}^d$.

Now, we  take $T=[0,1]^d$ with $\rho(s,t)=\max_{1 \leq i \leq d} |s_i - t_i|$, $D_n=\big\{2^{-n} u: u \in [\mathbf{1}, 2^n\mathbf{1}] \cap
\mathbb{Z}^d\big\}$ and apply Proposition~\ref{propn:xiao_2010} above to a self-similar $\sas$ random field $\{X(t)\}_{t \in
\mathbb{R}^d}$ whose first order increments are stationary. Using the self-similarity of $\{X(t)\}_{t \in \mathbb{R}^d}$, it follows (see
 the proof of Theorem \ref{label-result in case of tightness} 
 below) that for all $\gamma \in (0, \alpha \wedge 1)$ and for all $n \geq 1$,
\begin{equation} \label{term_exp:to_be_estimated}
\begin{split}
&\E \left(\displaystyle{\max_{\tau_n \in D_n}\,\max_{\tau^\prime_{n-1} \in O_{n-1}(\tau_n)}}|X(\tau_n)-X(\tau^\prime_{n-1})|^\gamma\right) \\
& \leq 2^{-nH\gamma} \displaystyle{\sum_{v \in V} \E\left(\max_{t \in [\mathbf{1}, 2^n \mathbf{1}] \cap \mathbb{Z}^d}|Y^{(v)}(t)|^\gamma\right)},
\end{split}
\end{equation}
where $\mathbf{Y}^{(v)}=\{Y^{(v)}(t)\}_{t \in \mathbb{Z}^d}$ is the discrete-parameter increment field defined by
$$
Y^{(v)}(t)=X(t+v)-X(t), \;\;t \in \mathbb{Z}^d
$$
in the direction $v \in V:=\{-1, 0, 1\}^d \setminus \{\mathbf{0}\}$.

The crucial observation is that due to the stationarity of the increments, each discrete-parameter field $\mathbf{Y}^{(v)}$ is stationary.
Therefore, in order to estimate the quantity in \eqref{term_exp:to_be_estimated}, it suffices to establish sharp upper bounds on
\begin{equation}\label{seq_maximal_moment}
\E\bigg(\max_{t \in [\mathbf{0}, (2^n-1) \mathbf{1}] \cap \mathbb{Z}^d}|Y(t)|^\gamma\bigg),
\end{equation}
where $\mathbf{Y}=\{Y(t)\}_{t \in \mathbb{Z}^d}$ is a stationary $\sas$ random field, $n \geq 1$ and $\gamma \in
(0, \alpha \wedge 1)$. This translates an investigation of sample path regularity properties into an extreme
value theoretic question.  Along this direction, some partial results were obtained in \cite{xiao:2010} 
which are applicable to stable random fields with certain specific dependence structures.
\subsection{Related works on partial maxima of stable processes}
In this work, we have improved upon the results in \cite{xiao:2010} and computed the exact rate of growth of the maximal moment sequence
\eqref{seq_maximal_moment} for a large class of stationary $\sas$ random fields, thus solving the problem of characterizing path properties of such random fields as posed in
\cite{xiao:2010} (see pages 173-174 therein). The main tools used in our solution  are ergodic-theoretic and algebraic in
nature as described below. We provide an overview of these techniques and related work below.

It was established by \cite{rosinski:1995, rosinski:2000} that every stationary $\SaS$ random field $\mathbf{Y}=\{Y(t)\}_{t \in \Zd}$
has an integral representation of the form
\begin{eqnarray}
Y(t) &\eqdef&\int_S c_t(s){\left(\frac{d \mu \circ \phi_t}{d
\mu}(s)\right)}^{1/\alpha}f \circ \phi_t(s) M(ds),\;\; \ \ t \in
\mathbb{Z}^d\,, \label{repn_integral_stationary}
\end{eqnarray}
where $M$ is a $\sas$ random measure on some standard Borel space $(S,\mathcal{S})$ with $\sigma$-finite
control measure $\mu$, $f \in L^{\alpha}(S,\mu)$, $\{\phi_t\}_{t \in \mathbb{Z}^d}$
is a nonsingular $\mathbb{Z}^d$-action on $(S, \mathcal{S},\mu)$ (i.e., each $\phi_t: S \to S$ is a measurable map,
$\phi_0$ is the identity map on $S$, $\phi_{u+v} = \phi_u \circ \phi_v$ for all $u, v \in \Zd$ and each $\mu \circ \phi_t$
 is equivalent to $\mu$), and $\{c_t\}_{t
\in \mathbb{Z}^d}$ is a measurable cocycle for $\{\phi_t\}$ (i.e., each $c_t$ is a $\{\pm 1\}$-valued measurable map
defined on $S$ satisfying $c_{u+v}(s)= c_u(\phi_v(s)) c_v(s)$ for all $u, v \in \Zd$ and for all $s \in S$). See, for
example, \cite{aaronson:1997}, \cite{krengel:1985}, \cite{varadarajan:1970} and \cite{zimmer:1984} for discussions
on nonsingular (also known as \emph{quasi-invariant}) group actions.

The Rosi\'nski Representation \eqref{repn_integral_stationary} is very useful in determining various probabilistic properties of
$\mathbf{Y}$; see, for example, \cite{mikosch:samorodnitsky:2000a}, \cite{samorodnitsky:2004a, samorodnitsky:2004b},
\cite{resnick:samorodnitsky:2004}, \cite{samorodnitsky:2005a}, \cite{roy:samorodnitsky:2008}, \cite{roy:2010a, roy:2010b},
 \cite{wang:2013}, \cite{chakrabarty:roy:2013},  \cite{fasen:roy:2016}. In this work, we shall focus on estimating the
maximal moment in (\ref{seq_maximal_moment}) and its connection to uniform modulus of continuity of  $\sas$ random fields.
We say that a stationary $\sas$ random field $\{Y(t)\}_{t\in \mathbb{Z}^d}$ is generated by a nonsingular
$\mathbb{Z}^d$-action $\{\phi_t\}$ on $(S, \mu)$ if it has an
integral representation of the form $(\ref{repn_integral_stationary})$ satisfying
the full support condition
$
\bigcup_{t \in \mathbb{Z}^d} \,\text{Support}( f \circ \phi_t)=S,
$
which will be assumed without loss of generality.

A measurable set $W \subseteq S$ is called a \emph{wandering set} for the nonsingular $\mathbb{Z}^d$-action
$\{\phi_t\}_{t \in \mathbb{Z}^d}$ if $\{\phi_t(W):\;t\in \mathbb{Z}^d\}$ is a pairwise
disjoint collection. The set $S$ can be decomposed into two
disjoint and invariant parts as follows: $S=\mathcal{C} \cup \mathcal{D}$, where $\mathcal{D} = \bigcup_{t \in \mathbb{Z}^d} \phi_t(W^\ast)$
for some wandering set $W^\ast \subseteq S$, and $\mathcal{C}$ has no wandering subset of positive $\mu$-measure; see \cite{aaronson:1997}
and \cite{krengel:1985}. This decomposition is called the {\em Hopf decomposition}, and the sets $\mathcal{C}$ and $\mathcal{D}$ are
called {\emph conservative} and {\emph dissipative} parts (of $\{\phi_t\}_{t \in G}$), respectively. The action is called conservative if
$S=\mathcal{C}$ and dissipative if $S=\mathcal{D}$.

Denote by $f_t(s)$ the family of functions on $S$ in the representation (\ref{repn_integral_stationary}):
\[
f_t(s) = c_t(s){\left(\frac{d \mu \circ \phi_t}{d
\mu}(s)\right)}^{1/\alpha}f \circ \phi_t(s),\;\; \ \ t \in
\mathbb{Z}^d.
\]
The Hopf decomposition of $\{\phi_t\}_{t \in \Zd}$ induces the following unique (in law) decomposition of the random field $\bY$
\begin{equation}
Y(t) \eqdef  \int_{\mathcal{C}} f_t(s)M(ds)+\int_{\mathcal{D}} f_t(s)M(ds):=Y^{\mathcal{C}}(t)+Y^{\mathcal{D}}(t),\;\; t \in \mathbb{Z}^d,
\label{decomp_of_X_t}
\end{equation}
where the two random fields $\bY^\mathcal{C}$ and $\bY^\mathcal{D}$ are independent and are generated by conservative
and dissipative $\Zd$-actions, respectively; see \cite{rosinski:1995, rosinski:2000}, and \cite{roy:samorodnitsky:2008}.
This decomposition reduces the study of stationary $\sas$ random fields to that of the ones generated by conservative
and dissipative actions.

It was argued by \cite{samorodnitsky:2004a} (see also \cite{roy:samorodnitsky:2008}) that stationary $\sas$
random fields generated by conservative actions have longer memory than those  generated by dissipative actions
and therefore, the following dichotomy were observed:
\begin{equation*}
n^{-d/\alpha} \max_{\|t\|_\infty \leq n}|Y(t)| \Rightarrow \left\{
                                     \begin{array}{ll}
                                     c_\bY Z_\alpha, & \mbox{ if $\bY$ is generated by a dissipative action,} \\
                                     0,              & \mbox{ if $\bY$ is generated by a conservative action}
                                     \end{array}
                              \right. 
\end{equation*}
as $n \rightarrow \infty$. In the limit above, $Z_\alpha$ is a standard Frech\'{e}t type extreme value random variable with distribution
\begin{equation}
\P(Z_\alpha \leq x)=e^{-x^{-\alpha}},\;\,x > 0,  \label{cdf_of_Z_alpha}
\end{equation}
and $c_\bY$ is a positive constant depending on the random field $\bY$.
In fact, this limiting behavior of the maximal process is closely tied with the
limit of the deterministic sequence
\begin{equation}\label{Def:bn}
\{b_n\}_{n\geq 1}
  =\left\{\left(\int_{S}\max_{{\bf 0}\leq t\leq (n-1)\bf{1}}|f_t(s)|^{\alpha}\mu(ds)\right)^{1/\alpha}\right\}_{n\geq 1},
\end{equation}
which has been proved by \cite{samorodnitsky:2004a}, \cite{roy:samorodnitsky:2008} to satisfy
\begin{equation}
\label{label-quantity that controls maximal functional}
n^{-d/\alpha} b_n \to \left\{
                                     \begin{array}{ll}
                                     \tilde{c}_\bY & \mbox{ if action is dissipative,} \\
                                     0              & \mbox{ if action is conservative,}
                                     \end{array}
                              \right. 
\end{equation}
where $\tilde{c}_\bY$ is a positive constant.\\
For conservative actions, the actual rate of growth of the partial maxima sequence $M_n$ depends on further properties
of the action as investigated in \cite{roy:samorodnitsky:2008}.

The work mentioned above hinges on some group theoretic preliminaries, as discussed briefly below. Let
$$A=\big\{\phi_{t}:\;t\in\mathbb{Z}^{d}\big\}$$
be a subgroup of the group of invertible nonsingular transformations on $(S,\mu)$ and define a group homomorphism,
$\Phi: \mathbb{Z}^{d}\to A $ by $\Phi(t)=\phi_{t}$ for all $t\in\mathbb{Z}^{d}$. Let
$$\mathcal{K}={\rm Ker}(\Phi)=\big\{t\in\mathbb{Z}^{d}:\phi_{t}=1_{S}\big\},$$
where $1_{S}$ denotes the identity map on $S$. Then $\mathcal{K}$ is a free abelian group and by the first isomorphism theorem of groups, we have
$$A\cong \mathbb{Z}^{d}/\mathcal{K}.$$
Now, by the structure theorem of finitely generated abelian groups (see, for example, Theorem~8.5 in Chapter~I of \cite{lang:2002}), we get,
$$A=\bar{F} \oplus \bar{N},$$
where $\bar{F}$ is a free abelian group and $\bar{N}$ is a finite group. Assume $rank(\bar{F})=p\geq 1$ and $|\bar{N}|=l$. Since,
$\bar{F}$ is free abelian, there exists an injective group homomorphism,
$$\Psi:\bar{F}\to \mathbb{Z}^{d},$$
such that $ \Phi\circ\Psi=1_{\bar{F}}.$ Then $F=\Psi(\bar{F})$ is a free subgroup of $\mathbb{Z}^{d}$ of rank $p$.
The subgroup $F$ can be regarded as an effective index set and its rank $p$ is an upper bound on effective dimension of the random field giving more precise information on the rate of growth of the partial maximum than the actual dimension $d$. Depending on the nature of the action restricted to $F$, the deterministic sequence $b_n$ controlling the rate of partial maxima shows
the following asymptotic behavior:
\begin{equation*}
\label{label-quantity that controls maximal functional:weak}
n^{-p/\alpha} b_n \to \left\{
                                     \begin{array}{ll}
                                     c & \mbox{ if action restricted to $F$ is not conservative,} \\
                                     0              & \mbox{ if action restricted to $F$ is conservative},
                                     \end{array}
                              \right. 
\end{equation*}
where $c$ is a positive and finite constant 
and $p$ is the effective dimension of the field as long as the restricted action is not conservative. Otherwise, $p$ should be regarded as an upper bound on the effective dimension.

More specifically, Theorem 5.4 in \cite{roy:samorodnitsky:2008} (which is summarized below) sharpens the description
of the asymptotic behavior of the partial maxima of
a random field when the action is conservative by observing the behavior of the action when restricted to the free subgroup $F$ of
$\mathbb{Z}^{d}$, leading to the conclusion that $\max_{\|t\|_\infty \leq n}|Y(t)|= O(n^{p/\alpha} )$ when the effective $F$-action is not conservative, and is $o(n^{p/\alpha})$ in the conservative case.
That is,
\begin{equation*}
\label{label-in prob. convergence gtp}
n^{-p/\alpha} \max_{\|t\|_\infty \leq n}|Y(t)| \Rightarrow \left\{
                                     \begin{array}{ll}
                                     c_\bY Z_\alpha & \mbox{ if $\{\phi_t\}_{t \in F}$ is not conservative,} \\
                                     0              & \mbox{ if $\{\phi_t\}_{t \in F}$ is conservative.}
                                     \end{array}
                              \right. 
\end{equation*}
Similar rates of growth are computed for continuous-parameter random fields in \cite{chakrabarty:roy:2013}
improving upon the works of \cite{samorodnitsky:2004b} and \cite{roy:2010b}.

\subsection{Our contributions}\label{sec:contri}
This work provides the rates of growth of the $\beta$-th moment of the partial maxima sequence denoted as
$$M_n = \max_{\B{0}\leq t\leq (n-1)\B{1}}|Y(t)|$$ whenever $0<\beta<\alpha$ for a stationary $\sas$ process
$\bY=\{Y(t)\}_{t\in \mathbb{Z}^d}$ with an integral representation given by \eqref{repn_integral_stationary} solving an open problem mentioned in \cite{xiao:2010}.
Theorem \ref{label-Discrete MoM} in Section \ref{sec:discrete} shows that the $\beta$-th moment of maxima of such
discrete random fields are $O(n^{d\beta/\alpha})$ for a nonconservative action and $o(n^{d\beta/\alpha})$ for a
conservative one. We sharpen the above asymptotics in the case of a conservative action by looking at properties
of the underlying action restricted to the free subgroup $F$ with effective dimension $p$; seeTheorem
\ref{label-Discrete Moments of maxima gtp}. 

We also introduce the concept of \emph{weak effective dimension}
generalizing the notion of effective dimension of \cite{roy:samorodnitsky:2008} and relate it to maximal moments
(see Theorem~\ref{weak:eff:dim} below) of stable random fields. We also provide easy extensions of our results
to the continuous parameter case in the appendix. The main idea of the proofs of these theorems is to exploit a series representation given in \cite{samorodnitsky:2004a} and follow the proof of its key result (see Theorem~4.1 therein) to get sharp tail bounds for the lower powers of maxima of stationary $\sas$ random fields so that dominated convergence theorem can be used.


Finally, we use the rates of growth of the partial maxima sequence for stationary random fields $\bY^{(v)}$ to derive path properties
of a real valued $H$-self-similar $\sas$ random field $\bX$ with stationary increments. Our main result is Theorem
\ref{label-result in case of tightness} which establishes uniform modulus of continuity for a large class of such random fields.
 As a consequence (see Corollary \ref{label-result in case of effective dimension}), we prove that the paths of $\bX$ are
 uniformly H\"older continuous of all orders $<H-\frac{p}{\alpha}$ when the corresponding increment processes $\bY^{(v)}$
 are generated by actions with effective dimension $p$. Corollary~\ref{label-result in case of weak effective dimension} connects path properties with weak effective dimension in a natural fashion. The short memory case (i.e., when the effective dimension $p=d$), on the other hand, is considered
 in Corollary \ref{label-result in case of dissipative flow}. 
These results show that in presence of stronger dependence $p < d$, the sample paths of $\bX$ become smoother because
stronger dependence prevent erratic jumps. Therefore, H\"older continuity of $\sas$ random fields also changes at the boundary
between short and long memory, which validates the conjecture in \cite[p.1440]{samorodnitsky:2004a}.


\section{Maximal Moments of Stationary $\sas$ Random Fields}\label{sec:discrete}

In this section, we solve an open problem mentioned in \cite[Page~173]{xiao:2010} and give sharp upper and lower bounds on maximal moments of stationary $\sas$ random fields when the maximum is taken over hypercubes of increasing size. Our results significantly improve the existing bounds given in Lemma~3.5 of \cite{xiao:2010} and hence the ones in \cite{talagrand2006generic}. This is achieved through exploitation of underlying nonsingular actions, and their ergodic theory and group theory. We also introduce the notion of weak effective dimension of stationary $\sas$ random fields in this section and apply it to estimate maximal moments.

The following is our main result on the asymptotic behavior of the maximal moments of stationary
$\sas$  random fields indexed by $\mathbb{Z}^d$. This result, together with the next two, solves the aforementioned open problem. The proofs are deferred to Section~\ref{sec:proofs} in order to increase the readability of our paper.

\begin{thm} \label{thm1:maximal_moment}
Let $\bY=\{Y(t)\}_{t\in \Zd}$ be a stationary $\sas$ random field with $0<\alpha< 2$ and having integral
representation as
\begin{equation}
\begin{split}
Y(t) &\eqdef \int_S f_t(s)M(ds)\\
&=\int_S c_t(s){\left(\frac{d \mu \circ \phi_t}{d
\mu}(s)\right)}^{1/\alpha}f \circ \phi_t(s) M(ds),\;\; t \in
\mathbb{Z}^d\,, \label{repn_integral_stationary2}
\end{split}
\end{equation}
where $M$ is a  $\SaS$ random measure on $(S, {\mathcal S})$ with a control measure $\mu$ as in
 (\ref{repn_integral_stationary}).
\begin{enumerate}
\item  If $\bY$ is generated by a dissipative action or equivalently\footnote{See Theorem 3.3 of
\cite{roy:samorodnitsky:2008}}, $\bY$ has a mixed moving
average representation given by
\begin{equation*}
\label{MMR}
\bY\stackrel{d}{=}\left\{\int_{W\times\mathbb{Z}^d}f(v,t+s)M(dv,ds)\right\}_{t\in\mathbb{Z}^{d}},
\end{equation*}
then, for $0<\beta<\alpha$,
\begin{equation}
\label{label-not conservative moments}
n^{{-d\beta }/{\alpha}} \E\big[M_{n}^{\beta}\big] \to C \ \ \text{ as } n \to \infty,
\end{equation}
where  $C=\tilde{c}_{\bY}^{\beta}C_{\alpha}^{\beta/\alpha} \E\left[Z_{\alpha/\beta}\right]$,
with $Z_{\alpha/\beta}$ denoting a Frech\'et random variable defined in
\eqref{cdf_of_Z_alpha} with shape parameter $\alpha/\beta$,
$\tilde{c}_{\bY}$ is the constant in  (\ref{label-quantity that controls maximal functional})
and
\begin{equation}
\label{C-alpha}
 C_\alpha=\begin{cases}
      \cfrac{1-\alpha}{\Gamma(2-\alpha)\cos(\pi\alpha/2)}, \ & \text{ if } \alpha\neq 1, \\
      \;\;\;\;\;\;\;\;\cfrac{2}{\pi}, & \text{ if } \alpha= 1.
        \end{cases}
\end{equation}
\item If $\bY$ is generated by a conservative action, then for $0<\beta<\alpha$,
\begin{equation}
\label{label-conservative moments}
n^{{-d\beta }/{\alpha}} \E\big[M_{n}^{\beta}\big] \to 0 \ \ \text{ as } n \to \infty. \\
\end{equation}
\end{enumerate}
\label{label-Discrete MoM}
\end{thm}


The above result solves an open problem mentioned (right after the proof of Lemma~3.5) in \cite{xiao:2010} when the underlying group action is dissipative. Note that as long as the action is not conservative, the same asymptotics will hold for the maximal moment sequence. In the next result, we present a solution to the problem in a more general situation. Before we describe the next theorem, we introduce the notion of weak effective dimension of a stationary $\sas$ random field. This notion should be considered significantly better than the effective dimension (defined by \cite{roy:samorodnitsky:2008}), which is always an integer whereas weak effective dimension need not be an integer.
\begin{defn}
\label{weak:eff:dim:defn}
We say that a stationary $\sas$ random field has weak effective dimension bounded by $\theta_2 \in (0, d]$ if there exist constants $c_1 >0,\, c_2>0$ and $\theta_1  \in (0,\theta_2]$ such that the sequence $b_n$ defined by \eqref{Def:bn} satisfies
\begin{equation}
c_1 n^{\theta_1} \leq b_n^\alpha \leq c_2 n^{\theta_2}  \label{weak:eff}
\end{equation}
for all sufficiently large $n$. If \eqref{weak:eff} is satisfied with $\theta_2=\theta_1$, then we call $\theta_2$ the weak effective dimension of the random field. If further weak effective dimension $\theta_2 = d$, then we say that the stationary $\sas$ random field is full-dimensional.
\end{defn}

Clearly, Proposition~4.1 in \cite{roy:samorodnitsky:2008} ensures that any stationary $\sas$ random field with a nontrivial dissipative (equivalently, mixed moving average) part is full-dimensional. The rationale behind this nomenclature (and also behind restricting the value of $\theta_2$ in the interval $(0, d]$) can be explained by the following calculation:
\begin{align}
b_n & =\left( \int_S \max_{{\bf 0}\leq t\leq (n-1)\bf{1}} |f_t(x)|^\alpha \mu(dx) \right)^{1/\alpha} \nonumber \\
       & = \left( \int_S \max_{{\bf 0}\leq t\leq (n-1)\bf{1}} \left[|f \circ \phi_t(x)|^\alpha \frac{d\mu \circ \phi_t}{d \mu}(x) \right]\mu(dx) \right)^{1/\alpha}. \nonumber \\
\intertext{Bounding the maximum by the sum and using Fubini's Theorem, we get}
b_n^\alpha & \leq  \left( \sum_{{\bf 0}\leq t\leq (n-1)\bf{1}}\int_S  \left[|f \circ \phi_t(x)|^\alpha \frac{d\mu \circ \phi_t}{d \mu}(x) \right]\mu(dx) \right)\nonumber \\
      &=\left( \sum_{{\bf 0}\leq t\leq (n-1)\bf{1}}\int_S |f \circ \phi_t(x)|^\alpha d\mu \circ \phi_t(x) \right)\nonumber \\
       & = \left( \sum_{{\bf 0}\leq t\leq (n-1)\bf{1}}\int_S  |f (x)|^\alpha d\mu(x)\right) = n^{d} \|f\|_\alpha^\alpha. \nonumber
\end{align}

If a stable random field has effective dimension (as described in Section~\ref{sec:prelim}) $p$, then thanks to Proposition~5.1 in \cite{roy:samorodnitsky:2008}, we can take $\theta_1 = \theta_2 = p$ in Definition~\ref{weak:eff:dim:defn} making this notion coincide with its weaker version introduced in Definition~\ref{weak:eff:dim:defn}. The connection of weak effective dimension to asymptotics of maximal moments is given in the following result, which also contributes to the solution of open problem mentioned in Page 173 of \cite{xiao:2010}.

\begin{thm}
\label{weak:eff:dim}
Consider a stationary $\sas$ random field with $0<\alpha< 2$, $\bY=\{Y(t)\}_{t\in \Zd}$ with integral representation as
\eqref{repn_integral_stationary2}. If the field has weak effective dimension bounded by $\theta_2$,
then for all $ n \ge 1$,
\begin{equation}
\label{weak:eff:bound}
n^{-\theta_2\beta / \alpha} \E\big[ M_n^{\beta}\big] \leq K^\prime,
\end{equation}
where $K^\prime$ is a finite constant.
\end{thm}

\begin{remark} \label{remark:LB_not_needed} \textnormal{By Theorem~2.1 of \cite{marcus:1984} (see also Equation (3.4) in \cite{samorodnitsky:2004a}), as long as $\alpha \in (0,1)$,}
\[
\E(M_n^\beta) \leq K_2 b_n^\beta
\]
\textnormal{always holds for some $K_2 \in (0, \infty)$, for all $\beta \in (0, \alpha)$ and for all $n \geq 1$. Therefore, the lower bound in \eqref{weak:eff} is not required when $0 < \alpha < 1$.}
\end{remark}


Now we consider the case when the underlying group action is conservative and establish refined results on maximal moments in  terms of the effective dimension $p$
of $\bY$. This is the place where algebra (more specifically, \emph{structure theorem for finitely generated abelian groups}) plays a significant role in the asymptotic properties of maximal moments and hence in solving the open problem in \cite{xiao:2010}.


\begin{thm} \label{thm3:maximal_moment}
\label{label-Discrete Moments of maxima gtp}
Let $\bY= \{Y(t)\}_{t\in \Zd}$  be  a stationary $\sas$ random field with $0<\alpha< 2$, with
integral representation written in terms of functions $\{f_{t}\}$ as in \eqref{repn_integral_stationary2}.
\begin{enumerate}
\item  If the underlying action $\{\phi_t\}_{t\in F}$ is dissipative when restricted to free subgroup $F$ with rank $p$, then
\begin{equation}
\label{label-G, not conservative moments}
n^{-p\beta/\alpha} \E\big[{M_{n}}^{\beta}\big] \to  C \ \ \text{ as } n \to \infty,
\end{equation}
where constant $C=c^{\beta}C_{\alpha}^{\beta/\alpha} \E\left[Z_{\alpha/\beta}\right]$, with $Z_{\alpha/\beta}$ denoting
a Frech\'et random variable 
defined in \eqref{cdf_of_Z_alpha} with shape parameter $\alpha/\beta$ and constant
$$c=\lim_{n\to\infty} n^{-p/\alpha}b_n, $$
and $C_\alpha$ as defined in \eqref{C-alpha}.
\item If the underlying action $\{\phi_t\}_{t\in F}$ is conservative when restricted to free subgroup $F$
with rank $p$, then
\begin{equation}
\label{label-G, conservative moments}
n^{{-p\beta }/{\alpha}} \E\big[M_{n}^{\beta}\big] \to 0 \ \ \text{ as } n \to \infty. \\
\end{equation}
\end{enumerate}
\end{thm}

\begin{remark}
  \textnormal{The asymptotic properties of maximal moments can easily be extended to stationary measurable symmetric
  $\alpha$-stable random fields indexed by $\mathbb{R}^d$. This can be done based on the works of \cite{samorodnitsky:2004b},
  \cite{roy:2010b} and \cite{chakrabarty:roy:2013}. Since the results (and the proofs) are similar to those presented in this section,
  we have included them (only $d=1$ case for simplicity of presentation) in Theorem \ref{cont:parameter:gom} in Appendix \ref{Appendix} below.}
\end{remark}


\section{Uniform Modulus of Continuity}\label{sec:path:properties}
This section combines the maximal moment estimates in Section~\ref{sec:discrete} with Proposition \ref{propn:xiao_2010} to
establish uniform modulus of continuity of self-similar $\sas$ random fields with stationary increments. The basis behind the connection between maximal moments and path properties has already been explained in Section~\ref{sec:prelim} through a novel chaining argument of \cite{xiao:2010}.

As mentioned eaerlier, \cite{samorodnitsky:2004a} observed a phase transition boundary for the partial maxima sequence of stable processes that corresponds to an ergodic theoretic boundary (namely, the one obtained in Hopf decomposition) of the underlying nonsingular group action. He also conjectured that many other properties of stable processes will also undergo a phase transition at the same boundary. This conjecture has already been verified for  many probabilistic properties (e.g., ruin probabilities (see \cite{mikosch:samorodnitsky:2000a}), growth-rate of maxima sequence (see \cite{roy:samorodnitsky:2008}, \cite{roy:2010b}), point processes of extremes (see \cite{resnick:samorodnitsky:2004}, \cite{roy:2010a}), large deviations (see \cite{fasen:roy:2016}), statistical aspects (see \cite{bhattacharya:roy:2018}), etc.) of stable random fields but not for path behaviours.

The following theorem is the main result of this section and has three corollaries (see Corollaries \ref{label-result in case of dissipative flow}, \ref{label-result in case of effective dimension} and \ref{label-result in case of weak effective dimension} below) that confirm the aforementioned conjecture for uniform modulus of continuity of stable random fields.

\begin{thm}
\label{label-result in case of tightness}
Let $\bX=\big\{X(t)\big\}_{t \in \mathbb{R}^{d}}$ be a real-valued $H$-self-similar $\sas$ random field with
stationary increments and with the following integral representation
\begin{equation}
\label{label-integral representation: tightness}
X(t) \stackrel{d}{=}\int_{E}f_{t}(s)M(ds), \;\;\;t\in \mathbb{R}^{d},
\end{equation}
where $M$ is a $\sas$ random measure on a measurable space $(E,\mathcal{E})$ with a $\sigma$-finite
control measure $m$, while ${f}_{t}\in\; L^{\alpha}(m,\mathcal{E}) $ for all $t \in \mathbb{R}^{d}$.\\
Let $V=\left\{v=(v_{1},\cdots,v_{d}): v_{i}\in\left\{-1,0,1\right\}\right\} \backslash\{(0,\cdots, 0)\}$ be the set
of vertices of unit cubes in $[-1,1]^{d}$, excluding the origin ${\bf 0}$. Define for each $v\in V$, the random field
$\bY^{(v)} = \{Y^{(v)}(t), t\in \bbR^d\}$ by
 $Y^{(v)}\left(t\right)=X\left(t +v\right)-X\left(t \right),$
 with the integral representation given by $$Y^{(v)}(t)=\int_{E}f^{(v)}_{t}(x)M(dx),$$
where  $f^{(v)}_{t}=f_{v+t}-f_{t}$ for all $t \in \mathbb{R}^{d}$. 
If either
\begin{enumerate}
\item $0< \alpha < 1$ and there exist constants $0<\theta_2 < \alpha H$ and $K >0$ such that for all $v\in V$,
$$
 b^{(v)}_{n}:=\left(\int_{E}\max_{{\bf 0}\leq t\leq (n-1)\bf{1}}|f^{(v)}_{t}(x)|^{\alpha}m(dx)\right)^{1/\alpha} \leq K\, n^{\theta_2 / \alpha}
$$
for all sufficiently large $n$, or 
\item $1\leq \alpha < 2$ and there exists $\theta_{2} \in (0, \alpha H)$ such that for all $v\in V$ the increment field $\{Y^{(v)}\left(t\right)\}$ has weak effective dimension bounded by
$\theta_{2}$,
\end{enumerate}
then for any $0<\gamma< \alpha $, 
\begin{equation}
\label{label- path property 3}
\displaystyle\limsup_{h\to 0+}\dfrac{\sup_{t\in T}\sup_{|s-t|_\infty\leq h} |X(t)-X(s)|}{h^{(H-\theta_{2}/\alpha)}(\log{1/h})^{1/\gamma}}
=0\;\; \; \; {\rm a.s.},
\end{equation}
where $|s-t|_{\infty}=\max_{1\leq j\leq d}|s_{j}-t_{j}|$ is the $\ell^{\infty}$ metric on $\mathbb R^d$.
\end{thm}

\begin{proof} We first give the proof under condition $(2)$ (i.e., when $1\leq \alpha < 2$). In this case, define the sequence
 $\left\{D_{n}, n\geq 0 \right\}$ as,
$$
D_{n}=\left\{\left(\dfrac{k_{1}}{2^{n}},\dfrac{k_{2}}{2^{n}},\cdots,\dfrac{k_{d}}{2^{n}}\right):0\leq k_{j}\leq 2^{n} -1,
1\leq j\leq d\right\}.
$$
Then the sequence $\left\{D_{n}, n\geq 0 \right\}$  satisfies the assumptions
1) and 2) for a chaining sequence in Section 2.

Observe that for any $0<\gamma< \alpha $,
\begin{equation}
\label{bound-path property}
\begin{aligned}
& \mathbb{E}{\bigg(\displaystyle\max\limits_{\tau_{n}\in D_{n}}\displaystyle\max\limits_{\tau_{n-1}'\in O_{n-1}(\tau_{n})}
|X(\tau_{n})-X(\tau_{n-1}')|^{\gamma}\bigg)} \\
&\leq \sum_{v\in V}\mathbb{E}\Bigg(\displaystyle\max\limits_{0\leq k_{j}\leq 2^{n}-1,\forall j=1, \ldots, d}
\bigg \lvert X\bigg(\Big(\dfrac{k_{1}}{2^{n}},\cdots,\dfrac{k_{d}}{2^{n}}\Big)+\frac{v}{2^{n}}\bigg)  \\
& \;\;\;\;\;\;\;\;\;\;\;\;\;\;\;\;\;\;\;\;\;\;\;\;\;\;\;\;\;\;\;\;\;\;\;\;\;\;\;\;\;\;\;\;-X\bigg(\dfrac{k_{1}}{2^{n}},,\cdots,
\dfrac{k_{d}}{2^{n}}\bigg)\bigg\lvert^{\gamma}\Bigg)  \\
&= 2^{-n\gamma H}\sum_{v\in V}\mathbb{E}{\left(\displaystyle\max\limits_{0\leq k_{j}\leq 2^{n}-1,\forall j=1,\ldots, d}
\big|Y^{(v)}\left((k_{1},\cdots,k_{d})\right)|^{\gamma}\right)}   \\
&= 2^{-n\gamma H}\sum_{v\in V}\mathbb{E}\Big[\big(M^{(v)}_{2^{n}}\big)^{\gamma}\Big],
\end{aligned}
\end{equation}
where ${M^{(v)}}$ is the partial maxima sequence of the stationary $\sas$ random field $\bY^{(v)}$,
and where the first equality follows from the self-similarity of $\bX$. Under the assumption of Theorem \ref{label-result in case of tightness}
we have that for  some positive constants $\theta_1$, $\theta_2$, $c_1$ and $c_2$,
$$c_1 n^{\theta_1 / \alpha} \leq b^{(v)}_{n}\leq c_2 n^{\theta_2 / \alpha}.$$
It follows from Theorem~\ref{weak:eff:dim} that  the sequence $ \mathbb{E}\big[ \big({b^{(v)}_{n}}^{-1}{M^{(v)}_{n}}\big)^{\gamma}\big]$
is bounded above by a constant $K' >0$. Hence
$$
\mathbb{E}\left[\big(M^{(v)}_{n}\big)^{\gamma}\right] \leq K'\big( {b_n^{(v)}}\big)^{\gamma}
\leq K\,n^{\theta_{2} \gamma / \alpha}
$$
for a finite constant $K>0$.
Noting that the cardinality of $V$ is $|V|=3^{d}-1$, we have
$$
\mathbb{E}{\bigg(\displaystyle\max\limits_{\tau_{n}\in D_{n}}\displaystyle\max\limits_{\tau_{n-1}' \in O_{n-1}(\tau_{n})}
|X(\tau_{n})-X(\tau_{n-1}')|^{\gamma}\bigg)}\leq (3^{d}-1)K\, 2^{-n\gamma (H-\theta_{2} / \alpha)}.$$
It follows immediately from Proposition \ref{propn:xiao_2010} that for any $\epsilon > 0$ and $\gamma \in (0, \alpha)$,
$$\displaystyle\limsup_{h\to 0+}\dfrac{\sup_{t\in T}\sup_{|s-t|_\infty \leq h} |X(t)-X(s)|}
{h^{(H-\theta_{2}/\alpha) }(\log{1/h})^{(1+\epsilon)/\gamma}}=0\;\;a.s.
$$
Since $\epsilon > 0$ and $\gamma $ are arbitrary,  \eqref{label- path property 3}  follows.

Under condition $(1)$, the same proof will go through because when $0 < \alpha < 1$, the lower bound on
$ b^{(v)}_{n}$ is not needed for establishing $\mathbb{E}\left[\big(M^{(v)}_{n}\big)^{\gamma}\right]
\leq K\,n^{\theta_{2} \gamma / \alpha}$ for some $K >0$ (see Remark~\ref{remark:LB_not_needed} above).
This completes the proof of Theorem \ref{label-result in case of tightness}.
\end{proof}

The above theorem has three important consequences (see below) that describe how the uniform modulus of continuity changes for self-similar stable random fields with stationary increments as we pass from a dissipative action to a conservative one in the integral representation of the increment fields. The more is the strength of conservativity of the action, the lower is the value of the (weak) effective dimension of the increment fields and the smoother are the paths of the original field due to longer memory. This is the heuristic reason why the phase-transition conjecture of \cite{samorodnitsky:2004a} can be verified (through Corollaries \ref{label-result in case of dissipative flow}, \ref{label-result in case of effective dimension} and \ref{label-result in case of weak effective dimension}) for path properties of stable random fields.


\begin{cor}
\label{label-result in case of dissipative flow}
Let $\bX=\big\{X(t)\big\}_{t \in \mathbb{R}^{d}}$ be a real-valued $H$-self-similar $\sas$ random field with stationary
increments and with the integral representation \eqref{label-integral representation: tightness}.
If, for every vertex $v \in V$, the increment process $\bY^{(v)}$ 
defined as in Theorem \ref{label-result in case of tightness} is generated by a dissipative action and $\alpha>\dfrac{d}{H}$,
then for any $0<\gamma< \alpha$,
\begin{equation}
\label{label- path property 1}
\displaystyle\limsup_{h\to 0+}\dfrac{\sup_{t\in T}\sup_{|s-t|_\infty \leq h}|X(t)-X(s)|}{h^{(H-d/\alpha)}(\log{1/h})^{1/\gamma}}=0\;\;a.s.
\end{equation}
\end{cor}

\begin{proof} Considering the same chaining sequence as in the proof of Theorem \ref{label-result in case of tightness}
with the $\ell^{\infty}$ metric, we may proceed similarly as in \eqref{bound-path property} to derive that for any $0 < \gamma < \alpha$,
\begin{equation*}
\mathbb{E}\bigg(\displaystyle\max\limits_{\tau_{n}\in D_{n}}\displaystyle\max\limits_{\tau_{n-1}^{'}\in O_{n-1}(\tau_{n})}
|X(\tau_{n})-X(\tau_{n-1}^{'})|^{\gamma}\bigg) \leq 2^{-n\gamma H}\sum_{v\in V}\mathbb{E}\left[\big(M^{(v)}_{2^{n}-1}\big)^{\gamma} \right],
\end{equation*}
where ${M^{(v)}}$ is the partial maxima sequence of the stationary $\sas$ random field $\bY^{(v)}$.
From Theorem \ref{label-Discrete MoM} in Section \ref{sec:discrete}, when $\bY^{(v)}$ is generated by a dissipative action, we have
\begin{equation}
\label{label- compuation of rate of growth}
\lim_{n\to\infty}\mathbb{E}\Big[(2^{n}-1)^{-\gamma{d/{\alpha}}} {\big(M^{(v)}_{2^{n}-1}\big)^{\gamma}}\Big]= c,
\end{equation}
where $c>0$ is a finite constant. Hence, there exists a finite constant $K$ such that
$$
\mathbb{E}{\left(\displaystyle\max\limits_{\tau_{n}\in D_{n}}\displaystyle\max\limits_{\tau_{n-1}^{'}\in O_{n-1}(\tau_{n})}
|X(\tau_{n})-X(\tau_{n-1}^{'})|^{\gamma}\right)}\leq K\, 2^{-n\gamma (H-d/\alpha)}, 
$$
for all sufficiently large $n$. It is now clear that \eqref{label- path property 1}  follows from Proposition \ref{propn:xiao_2010}.
\end{proof}


\begin{cor}
\label{label-result in case of effective dimension}
Let $\bX=\big\{X(t)\big\}_{t \in \mathbb{R}^{d} }$ be a real-valued $H$-self-similar random field with stationary increments as in Theorem
\ref{label-result in case of tightness}. If, for every vertix $v \in V$, the increment process $\bY^{(v)}$
has effective dimension $p\leq d$ and $\alpha>\dfrac{p}{H}$, then for any $0<\gamma< \alpha,$
\begin{equation*}
\label{label- path property 4}
\displaystyle\limsup_{h\to 0+}\dfrac{\sup_{t\in T}\sup_{|s-t|_\infty \leq h}|X(t)-X(s)|}{h^{(H-p/\alpha)}(\log{1/h})^{1/\gamma}}=0\;\;a.s.
\end{equation*}
\end{cor}
\begin{proof}
The proof follows similarly along the lines of Corollary \ref{label-result in case of dissipative flow}
by using the bound on moments in terms of the effective dimension in Theorem \ref{label-Discrete Moments of maxima gtp}.
\end{proof}


\begin{cor}
\label{label-result in case of weak effective dimension}
Let $\bX=\big\{X(t)\big\}_{t \in \mathbb{R}^{d} }$ be a real-valued $H$-self-similar random field with stationary increments as in Theorem
\ref{label-result in case of tightness}. If for every vertex $v \in V$, the increment field $\bY^{(v)}$
has weak effective dimension bounded by $\theta_2 \in (0, \alpha H)$, then for any $0<\gamma< \alpha$,
\begin{equation}
\label{label- path property 5}
\displaystyle\limsup_{h\to 0+}\dfrac{\sup_{t\in T}\sup_{|s-t|_\infty \leq h}|X(t)-X(s)|}{h^{(H-\theta_2/\alpha)}(\log{1/h})^{1/\gamma}}=0\;\;a.s.
\end{equation}
\end{cor}
\begin{proof}
The proof follows immediately by using the same arguments as the second part of Theorem \ref{label-result in case of tightness}.
\end{proof}

\begin{remark} \textnormal{(i)~If the (weak) effective dimension of the increment fields in Corollaries \ref{label-result in case of effective dimension} and \ref{label-result in case of weak effective dimension} are strictly less than $d$ (i.e., when we are not in the full-dimensional case), our uniform modulus of continuity results improve the leading (polynomial) term of the existing ones (see, for example, \cite{xiao:2010} and the references therein). On the other hand, in the full-dimensional case (i.e., in Corollary~\ref{label-result in case of dissipative flow}), we better the logarithmic term in the modulus of continuity.}

\textnormal{(ii)~From the proof of Corollary~\ref{label-result in case of weak effective dimension}, it transpires that even when the weak effective dimension of  $\bY^{(v)}$ is bounded by $\theta_2(v)$ (that may depend on $v \in V$), \eqref{label- path property 4} holds with $\theta_2$ replaced by $\max_{v \in V} \theta_2(v)$ as long as this maximum is strictly less than $\alpha H$. A similar comment applies to Corollary~\ref{label-result in case of effective dimension} above.}
\end{remark}

\section{Examples}\label{sec:examples}

The theorems in Sections 5 can be applied to various classes of self-similar random fields with stationary increments.
In the following, we mention two examples of them: linear fractional stable motion and harmonizable fractional stable motion.
We refer to  \cite{samorodnitsky:taqqu:1994} for more information on these two classes of important self-similar stable processes.
For further examples of self-similar processes with stationary increments, see \cite{pipiras:taqqu:2002a}.

\subsection{Linear fractional stable motion}

For any given  constants $0 < \alpha < 2$ and $H \in (0,
1)$, we define a $\sas$ process $Z^{H} = \{Z^{
H}(t)\}_{ t \in \bbR_+}$ with values in $\bbR$ by
\begin{equation}\label{Eq:Rep-LFSS}
Z^{H}(t) =  \kappa\, \int_{\bbR} \Big\{(t-s)_+^{H-1/\alpha} -
(-s)_+^{H-1/\alpha} \Big\} \, M_\alpha(ds),
\end{equation}
where $\kappa > 0$ is a normalizing constant, $t_{+} = \max\{t, 0\}$ and
$M_\alpha$ is a $\sas$ random measure with
Lebesgue control measure. 

Using (\ref{Eq:Rep-LFSS}) one can verify that the stable process
$Z^{H}$ is $H$-self-similar and
has stationary increments. It is a stable analogue of fractional
Brownian motion, and it called a linear fractional stable motion (LFSM).

Many sample path properties of $Z^{H}$ are different from those of
fractional Brownian motion. For example, \cite{maejima:1983} showed that,
if $H \alpha < 1$, then $Z^H$ has a. s. unbounded sample functions
on all intervals. \cite{takashima:1989} showed that, if $ H \alpha > 1$, then
the index of uniform H\"older continuity of $Z^H$ is $H - \frac 1 \alpha$.

In order to apply the results in Section 5, we consider for every $v \in \{-1, 1\}$ the increment process
\begin{equation}\label{Eq:Rep-Y1}
Y^{(v)}(t) = \int_{\bbR} \Big\{(t+v-s)_+^{H-1/\alpha} -
(t-s)_+^{H-1/\alpha} \Big\} \, M_\alpha (ds).
\end{equation}
Then for any $n \ge 1$,
\begin{equation*}
\label{Eq:bnL}
b_n^{(v)} = \left(\int_{\mathbb R} \max_{0 \leq k\leq n-1} \Big| (k+v -s)_+^{H - 1/\alpha} - (k -s)_+^{H - 1/\alpha}\Big|^{\alpha} \, ds\right)^{1/\alpha}.
\end{equation*}
For simplicity, we only consider the case of $v = 1$ and write $b_n^{(v)}$ as $b_n$. The case of $v = -1$ can be treated the same way.
For integers $k \in \{0, 1, \ldots, n-1\}$, let $g_k(s) = (k+1 -s)_+^{H - 1/\alpha} - (k -s)_+^{H - 1/\alpha}$. It is easy to see that
for each fixed $s \le k$, the sequence $g_k(s)$ is non-negative and non-increasing in $k$. We write $b_n^\alpha$ as
\begin{equation}
\label{Eq:bnL1}
\begin{split}
b_n^\alpha &= \int_{-\infty}^0  \max_{0 \leq k\leq n-1} g_k(s)^\alpha ds + \sum_{\ell = 0}^{n-1} \int_{\ell}^{\ell +1}  \max_{0 \leq k\leq n-1} g_k(s)^\alpha ds\\
& =   \int_{-\infty}^0 g_0(s)^\alpha ds + \sum_{\ell = 0}^{n-1} \int_{\ell}^{\ell +1}  \max \Big\{g_{\ell }(s), g_{\ell+1}(s)\Big\}^\alpha ds,
\end{split}
\end{equation}
where $g_n\equiv 0$. Now it is elementary to verify that each of the $(n+1)$ integrals in the right hand side of \eqref{Eq:bnL1} is
a positive and finite constant depending only on $\alpha$ and $H$. Except the first and the last integrals,  all the other integrals
are equal. Consequently, there is a positive and finite constant $K$  such that
\[
\lim_{n \to \infty} n^{-1/\alpha}  b_n = K.
\]
Hence, for $v\in \{1, -1\}$, the weak effective dimension of the stationary $\sas$ process $\big\{Y^{(v)} (n)\big\}_{n \in {\bbZ} }$
is 1. It can be verified that
\[
\sum_{k \in {\bbZ}} \big|g_k(s)\big|^\alpha < \infty \quad \ \hbox{ for a.e. } s \in \mathbb R.
\]
It follows from Corollary 4.2 of Rosi\'nski (1995) that $\big\{Y^{(v)} (n)\big\}_{n \in {\bbZ} }$ is generated by a dissipative flow.
Moreover, Condition (2) of Theorem 4.1 is satisfied with $\theta_1 = \theta_2 = 1$. It follows
from \eqref{label- path property 5}  that, if $H > 1/\alpha$ then  for
any $0<\gamma< \alpha$,
\begin{equation}
\label{LFSM-1}
\displaystyle\limsup_{h\to 0+}\dfrac{\sup_{t\in [0, 1]}\sup_{|s-t|\leq h}|Z^H(t)-Z^H(s)|}{h^{(H-1/\alpha)}(\log{1/h})^{1/\gamma}}=0\;\;a.s.
\end{equation}

This result improves Theorem 2 in \cite{kono:maejima:1991a}. We mention that, by using more delicate analysis,  \cite{takashima:1989}
established the exact uniform and local moduli of continuity of linear fractional stable motion $Z^H$ with $H > 1/\alpha$, and \cite{Balanca}
studied the multifractal property of $Z^H$ by using the 2-microlocal formalism.


\subsection{Harmonizable fractional stable motion}

For any given $\alpha \in (0,2)$ and  $H \in (0, 1)$, let $ \widetilde{Z}^{H}= \{\widetilde{Z}^{H}(t)\}_{t \in \bbR}$ be the real-valued
harmonizable fractional $\sas$ process (HF$\alpha$SF or HFSF, for brevity) with
Hurst index $H$, defined by:
\begin{equation}
\label{Eq:defhfsf}
\widetilde{Z}^{H}(t) :=  \widetilde \kappa\, {\rm Re} \int_{\bbR} \frac{ e^{it x}-1} {|x|^{H+1/\alpha}} \,\widetilde{M}_\alpha(dx),
\end{equation}
where 
$\widetilde \kappa$ is the positive normalizing constant given by
\begin{equation}\label{kappa}
\widetilde \kappa= 2^{-1/2} \bigg(\int_{\bbR}
 \frac{\big(1 - \cos x\big)^{\alpha/2}}{|x|^{\alpha H +1}}\, dx\bigg)^{-1/\alpha},
\end{equation}
Re denotes the real-part, and  $\widetilde{M}_\alpha$ a complex-valued rotationally invariant $\alpha$-stable random measure
with Lebesgue control measure.

For every $v \in \{-1, 1\}$ consider the increment process
\begin{equation}\label{Eq:Rep-Y2}
\widetilde{Y}^{(v)}(t) =  \widetilde \kappa\, {\rm Re} \int_{\bbR}  \frac{ e^{i (t+v) x}-e^{itx}} {|x|^{H+1/\alpha}} \,\widetilde{M}_\alpha(dx).
\end{equation}
Then for any integer $n \ge 1$,
\begin{equation*}
\label{Eq:bnH}
\begin{split}
b_n^{(v)} &=  \widetilde \kappa \left(\int_{\mathbb R} \max_{0 \leq k\leq n-1} \Big|  e^{i (k+v) x}-e^{i kx} \Big|^{\alpha} \,
\frac{dx}{|x|^{\alpha + H} }\right)^{1/\alpha}\\
&=  \widetilde \kappa \left(\int_{\mathbb R} \big| e^{i v x}- 1\big|^\alpha \, \frac{dx} {|x|^{\alpha + H} }\right)^{1/\alpha},
\end{split}
\end{equation*}
which is independent  of $n$. This implies that the weak effective dimension of the stationary $\sas$ process $\big\{
\widetilde Y^{(v)} (n)\big\}_{n \in {\bbZ} }$ is 0. Applying again Corollary 4.2 of Rosi\'nski (1995), one can verify that $\big\{\widetilde
Y^{(v)} (n)\big\}_{n \in {\bbZ} }$ is generated by a conservative flow.

We remark that the results in Section 3 are not applicable for determining the magnitude of
the maximal moments ${\mathbb E} \big[\max\limits_{0 \le k \le n-1} |\widetilde Y^{(v)}(k)|^\gamma\big]$ for
$\gamma \in (0, \alpha).$ By appealing to the fact that $\widetilde{Y}^{(v)}(t)$ is conditionally Gaussian, see
\cite{bierme2009holder, bierme2015modulus}, or \cite{kono:maejima:1991a}, we can modify the proof of Proposition
4.3 in  \cite{xiao:2010} to derive the following upper and lower bounds 
\begin{equation}
\label{Eq:mmH}
K \le {\mathbb E} \Big[\max_{0 \le k \le n-1} |Y^{(v)}(k)|^\gamma\Big] \le K' \big(\log n\big)^{\gamma/2},
\end{equation}
where $K$ and $K'$ are positive and finite constants. We omit a detailed verification of \eqref{Eq:mmH} here
because it is lengthy and does not produce the optimal bounds. We believe that the upper bound in \eqref{Eq:mmH}
is optimal. In the case of $\alpha = 2$, this can be proved by applying the Sudakov minoration
(see Lemma 2.1.2 in \cite{talagrand2006generic}).

It follows from \eqref{Eq:mmH}, \eqref{bound-path property} and Proposition \ref{propn:xiao_2010} with $\sigma(h)
= h^H \big|\log 1/h\big|^{1/2}$ that for any $\epsilon > 0$,
\begin{equation}
\label{HFSM-1}
\displaystyle\limsup_{h\to 0+}\dfrac{\sup_{t\in [0, 1]}\sup_{|s-t|\leq h}|\widetilde Z^H(t)- \widetilde Z^H(s)|}
{h^{H}(\log{1/h})^{\frac 1 2 +\frac {1} \alpha + \epsilon}}=0\;\;a.s.
\end{equation}
This recovers Theorem 1 in \cite{kono:maejima:1991a}. However, it is an open problem to
determine the exact uniform modulus of continuity for HFSM $ \widetilde{Z}^{H}$.

Even though LFSM $Z^{H}$ and HFSM $ \widetilde{Z}^{H}$ are both $H$-self-similar with stationary increments,
their properties are very different.
By the exact modulus of continuity in \cite{takashima:1989} and \eqref{HFSM-1}, it is clear that the laws of  $Z^{H}$
and $\widetilde{Z}^{H}$ are singular with respect to each other.

\section{Proofs from Section~\ref{sec:discrete}} \label{sec:proofs}
In this section, we present the proofs of the three theorems from Section~\ref{sec:discrete} that solve the open problem mentioned in the paper of \cite{xiao:2010}. The key idea is to encash a series representation given in \cite{samorodnitsky:2004a} (and follow the proof of Theorem~4.1 therein) to obtain sharp tail bounds for the lower powers of maxima of stationary $\sas$ random fields and then invoke dominated convergence theorem.

\subsection{Proof of Theorem~\ref{thm1:maximal_moment}}

In the following,  $\{\Gamma_n\}_{n\ge 1}$ denotes  a sequence of arrival times of a unit rate Poisson process on
$(0, \infty)$,  $\{\xi_n\}_{n\ge 1}$ are i.i.d. Rademacher random variables, and $\{U_\ell^{(n)}\}_{n\ge 1}$ ($\ell = 1, 2$)
are i.i.d. $S$-valued random variables with common law $\eta_n$ whose density is given by
\[
\frac{d\eta_n} {d \mu} = b_n^{-\alpha} \max_{{\bf 0}\leq t\leq (n-1)\bf{1}}|f_t(s)|^{\alpha}, \quad s \in S.
\]
All four sequences are independent. We will make use of the following series representation for
$\{Y_{k}, {\bf 0} \leq k\leq (n-1)\bf{1} \}$:
\[
Y_{k} \eqdef b_n C_\alpha^{1/\alpha}\sum_{j=1}^\infty \xi_j \Gamma_j^{-1/\alpha} \frac{f_k (U_j^{(n)})}
{\max_{m\in [{\bf{0}},(n-1)\bf{1}]}|f_m(U_j^{(n)})|}.
\]
See Section 3.10 in \cite{samorodnitsky:taqqu:1994}.

We first consider  case (1)  when $\bY$ is generated by a dissipative action, it follows from
 (\ref{label-quantity that controls maximal functional})  that the deterministic sequence
 $ \{b_n\}_{n\geq 1}$ satisfies
 \begin{equation}
 \label{lim-maximal}
 \lim_{n\to\infty} n^{-d/\alpha}b_n=\tilde{c}_{\bY},
\end{equation}
where $\tilde{c}_{\bY} > 0$ is a constant.
This implies that $b_n$ satisfies condition (4.6) in
\cite{samorodnitsky:2004a}, namely
\begin{equation*}
\textbf{(LB):} \;\;\;\;b_n \geq c n^{\theta} \text{ for some constant } c>0
\end{equation*}
with $\theta=d/\alpha$. Additionally, its
condition (4.8) given by
\begin{equation*}
\textbf{(LL):} \;\; \mathbb{P}\left[\text{ for some } k \in [{\bf{0}}, (n-1){\bf 1}],\;\frac{f_k (U_j^{(n)})}
{\max_{m\in [{\bf{0}},(n-1)\bf{1}]}|f_m(U_j^{(n)})|}, \; j=1,2\right] \to 0
\end{equation*}
as $n\to \infty$
also holds; thanks to Remark 4.2 in \cite{samorodnitsky:2004a} (or Remark 4.4 in
\cite{roy:samorodnitsky:2008}).
Further, (\ref{lim-maximal}) implies that for any $ p> \alpha$, there is a finite constant $A$ such that
$$
\textbf{(UB):}   \qquad\qquad
 n^{d} b_{n}^{-p}< n^{d} b_{n}^{-\alpha} \leq A.$$
Let $K=d$, $\epsilon$ and $\delta$ be chosen such that
\begin{equation*}
\label{choice-ep}
0< \epsilon < \dfrac{\delta}{K}.
\end{equation*}
Then we obtain from (4.21) in \cite{samorodnitsky:2004a} the following upper bound on the tail distribution of
$b_n^{-1} M_n$ under \textbf{(LB)} and \textbf{(LL)}:
\begin{equation}
\label{bound-ini}
\P \big(b_n^{-1} M_n> \lambda \big) \leq \P \big(C_{\alpha}^{1/\alpha}\Gamma_1^{-1/\alpha}>\lambda(1-\delta) \big)
+\phi_n (\epsilon,\lambda)+\psi_n (\epsilon,\delta,\lambda).
\end{equation}
Below, we explain the terms $\phi_n (\epsilon,\lambda)$ and $\psi_n (\epsilon,\delta,\lambda)$ used in
the above equation and derive some (preliminary) upper bounds on them.\\

\begin{equation}
\label{phi}
\begin{split}
\phi_n (\epsilon,\lambda)&= \P \bigg( 
\exists\, k\in [{\bf{0}},(n-1){\bf 1}], \, \;\;
\dfrac{ \Gamma_{j}^{-1/\alpha} |f_k(U_j^{(n)})|}{\max_{m\in [{\bf{0}},(n-1)\bf{1}]}|f_m(U_j^{(n)})|}>
\frac{\epsilon \lambda }{C_\alpha^{1/\alpha}} \\
& \qquad \qquad \qquad \qquad \qquad \;\;\text{ for at least 2 different } j \bigg) \\
&\le n^d \P\bigg( \Gamma_{j}^{-1/\alpha}> \frac{b_n\epsilon \lambda}{C_{\alpha}^{1/\alpha}\|f\|_{\alpha}}
 \; \text{ for at least 2 different } j \bigg).
\end{split}
\end{equation}
In deriving the last inequality, we have applied the fact that for every $k\in [{\bf{0}},(n-1){\bf{1}}] $, the points
$$
b_n \xi_j \Gamma_j^{-1/\alpha}\cfrac{f_k(U_j^{(n)})}{\max_{{\bf{0}}\leq s\leq (n-1)\bf{1}}|f_s(U_j^{(n)})|}, \quad j = 1, 2 , \cdots
$$
have the same  joint distribution as the points
$$ \xi_j \|f\|_{\alpha}\Gamma_j^{-1/\alpha},  \quad j = 1, 2 , \cdots $$
which represent a symmetric Poisson random measure on $\mathbb{R}$ with mean measure
\begin{equation}
\label{mean-measure-discrete}
\Lambda((x,\infty))=x^{-\alpha}\|f\|_{\alpha}^{\alpha}/2,\ \  \text{ for } x> 0.
\end{equation}
In the above, the function $f$ is given in (\ref{repn_integral_stationary2}) and $\|f\|_{\alpha}
 = \left(\int_S |f(s)|^\alpha\mu(ds)\right)^{1/\alpha}$. Similarly,
we have
\begin{equation}
\label{psi}
\begin{split}
\psi_n (\epsilon,\delta,\lambda) &= \P\bigg(\max_{k\in [{\bf{0}},(n-1)\bf{1}] }\bigg\lvert \sum_{j=1}^{\infty}
\dfrac{\xi_j \Gamma_j^{-1/\alpha} |f_k(U_j^{(n)})|}{\max_{m\in [{\bf{0}},(n-1)\bf{1}]}|f_m(U_j^{(n)})|}\bigg\lvert >
\frac{ \lambda} {C_{\alpha}^{1/\alpha}\|f\|_{\alpha}}, \ \nonumber\\
& \qquad\qquad\qquad  \Gamma_1^{-1/\alpha}\leq \frac{b_n\lambda(1-\delta)}{C_{\alpha}^{1/\alpha}\|f\|_{\alpha}}, 
\text{ and } \Gamma_2^{-1/\alpha}\leq \frac{b_n \lambda \epsilon}{C_{\alpha}^{1/\alpha}\|f\|_{\alpha}} \bigg)\nonumber \\
&\leq  n^{d}\P\bigg(\Big\lvert \sum_{j=1}^{\infty}\xi_j\Gamma_j^{-1/\alpha}\Big\lvert >  \frac{b_n \lambda}
{C_{\alpha}^{1/\alpha}\|f\|_{\alpha}}, \
   \Gamma_1^{-1/\alpha}\leq \frac{b_n\lambda(1-\delta)}{C_{\alpha}^{1/\alpha}\|f\|_{\alpha}}, \nonumber \\
& \qquad \qquad \qquad \qquad \qquad\qquad\qquad \;\;\;\;\;  \text{ and } \Gamma_2^{-1/\alpha}\leq \frac{b_n \lambda \epsilon}
{C_{\alpha}^{1/\alpha}\|f\|_{\alpha}} \bigg).
\end{split}
\end{equation}

For any $0<\beta<\alpha$, by using the tail bound in \eqref{bound-ini} we have
\begin{equation}
\label{MoM-1}
\begin{split}
\E\big[b_n^{-\beta} M_n^{\beta} \big] &= \int_0^\infty \P \big(b_n^{-1} M_n> \tau^{1/\beta}\big)d\tau   \\
&\leq   \int_0^\infty \P\big(C_{\alpha}^{1/\alpha}\Gamma_1^{-1/\alpha}>\tau^{1/\beta}(1-\delta)\big)d\tau  \\
&\qquad +  \int_0^\infty \phi_n (\epsilon,\tau^{1/\beta})d\tau +\int_0^\infty \psi_n (\epsilon,\delta,\tau^{1/\beta})d\tau   \\
&:= T_1(\delta) + T_2^{(n)}(\epsilon)  + T_3^{(n)}(\epsilon, \delta).
\end{split}
\end{equation}
It is shown in \cite{samorodnitsky:2004a} that for every $\tau>0$,
$$\phi_n (\epsilon, \tau^{1/\beta}), \text{ and } \psi_n (\epsilon,\delta,\tau^{1/\beta}) \text{ converge to } 0,$$
as $n\to \infty$ for choices of $\epsilon$ adequately smaller in comparison to $\delta$.\\

Next we present non-trivial integrable bounds on $(1,\infty)$ for integrands $\phi_n (\epsilon,\tau^{1/\beta})$ and
$\psi_n (\epsilon,\delta,\tau^{1/\beta})$ in $T_2^{(n)}(\epsilon) \text{ and } T_3^{(n)}(\epsilon, \delta)$ in \eqref{MoM-1}
 respectively, and use the trivial bound of $1$ on $(0,1)$. Finally, we apply the Dominated Convergence
 Theorem (DCT) to show that the terms
 $T_2^{(n)}(\epsilon) \text{ and } T_3^{(n)}(\epsilon, \delta)$ converge to $0$ as $n\to \infty$.

We begin by providing an integrable upper bound for $\phi_n (\epsilon,\tau^{1/\beta})$ on $(1,\infty)$. It follows from (\ref{phi})
that
\begin{equation}
\label{phi-bound}
\begin{split}
\phi_n (\epsilon,\tau^{1/\beta})
&\leq  n^d \P\bigg( \sum_{j=1}^\infty 1_{ \xi_j \|f\|_{\alpha}\Gamma_j^{-1/\alpha}}\Big\{ \big(-\infty,\, -  C_\alpha^{-1/\alpha} b_n\epsilon \tau^{1/\beta} \big) \\
& \qquad\qquad \qquad \qquad \qquad \;\cup \big( C_\alpha^{-1/\alpha} b_n \epsilon \tau^{1/\beta},\, \infty \big)\Big\}\geq 2\bigg)  \\
&= n^d \P(\text{Poi}(\Lambda(B_n))\geq 2),
\end{split}
\end{equation}
where
$$
B_n= \big(-\infty, -C_\alpha^{-1/\alpha}b_n\epsilon \tau^{1/\beta} \big)\cup \big( C_\alpha^{-1/\alpha}b_n\epsilon \tau^{1/\beta}, \infty\big)
$$
and we have used the fact that
$$
\sum_{j=1}^\infty 1_{ \xi_j \|f\|_{\alpha}\Gamma_j^{-1/\alpha}}\{B_n\}\sim \text{Poi}(\Lambda(B_n)).
$$
Thus, the Markov inequality and definition \eqref{mean-measure-discrete} of the mean measure $\Lambda$ imply
\begin{equation}
\label{phi-bound2}
\begin{split}
\phi_n (\epsilon,\tau^{1/\beta})  &\leq n^d \cfrac{\E(\text{Poi}(\Lambda(B_n)))}{2}= n^d \Lambda(B_n)/2   \\
&= n^d b_n^{-\alpha}\cfrac{C_\alpha^{-1}\epsilon^{-\alpha}}{\tau^{\alpha/\beta}}   \\
&\leq  A \cfrac{C_\alpha^{-1}\epsilon^{-\alpha}}{\tau^{\alpha/\beta}} \text{ using \textbf{(UB)}.}
\end{split}
\end{equation}
The last term in \eqref{phi-bound2} is clearly integrable in $\tau$ on $(1,\infty)$.
We apply DCT to $T_2^{(n)}(\epsilon)$ as
\begin{equation*}
\label{DCT_T2}
\begin{aligned}
T_2^{(n)}(\epsilon) &= \int_0^1 \phi_n (\epsilon,\tau^{1/\beta})d\tau +\int_1^\infty \phi_n (\epsilon,\tau^{1/\beta})d\tau
\end{aligned}
\end{equation*}
by using the trivial bound of $1$ on $(0,1)$ and the bound derived in \eqref{phi-bound2} on $(1,\infty)$ to conclude
\begin{equation*}
\label{T2-lim}
T_2^{(n)}(\epsilon)\to 0\ \text{ as } \ n\to \infty.
\end{equation*}
 We next derive an upper bound for $\psi_n (\epsilon,\delta,\tau^{1/\beta}) $. It follows from (\ref{psi}) that
 $\psi_n (\epsilon,\delta,\tau^{1/\beta})$ is  bounded from above by
\begin{equation}
\label{psi-bound}
\begin{split}
&  n^{d}\P\bigg(\Big\lvert C_{\alpha}^{1/\alpha}\sum_{j=1}^{\infty}\xi_j\Gamma_j^{-1/\alpha}\Big\lvert >
\frac{b_n\tau^{1/\beta}}{\|f\|_{\alpha}},  \
C_{\alpha}^{1/\alpha}\Gamma_1^{-1/\alpha}\leq \frac{b_n\tau^{1/\beta}(1-\delta)}{\|f\|_{\alpha}},  \\
& \qquad\qquad \;\;\;\;\;\; \text{ and } C_{\alpha}^{1/\alpha}\Gamma_j^{-1/\alpha}\leq \frac{b_n \tau^{1/\beta}\epsilon}{\|f\|_{\alpha}}
\text{ for all } j\geq 2 \bigg) \\
&\leq  n^{d}\P\bigg(C_{\alpha}^{1/\alpha}\Big|\sum_{j=K+1}^{\infty}\xi_j\Gamma_j^{-1/\alpha}\Big|>\frac{b_n\tau^{1/\beta}(\delta-\epsilon(K-1))}
{\|f\|_{\alpha}}\bigg)  \\
&\leq  n^{d}\P\bigg(C_{\alpha}^{1/\alpha}\Big|\sum_{j=K+1}^{\infty}\xi_j\Gamma_j^{-1/\alpha}\Big|
>\frac{b_n\tau^{1/\beta}\epsilon}{\|f\|_{\alpha}}\bigg)\\
&\leq  n^{d} b_{n}^{-p}\dfrac{\|f\|_{\alpha}^{p}\E \Big|C_{\alpha}^{1/\alpha}\displaystyle \sum\limits_{j=K+1}^{\infty}
{\xi_{j}\Gamma_{j}^{-1/\alpha} \Big|^{p}}}{\tau^{p/\beta}\epsilon^{p}}  \\
&\leq A \|f\|_{\alpha}^{p} \cfrac{\E \Big|C_{\alpha}^{1/\alpha}\displaystyle \sum\limits_{j=K+1}^{\infty}{\xi_{j}\Gamma_{j}^{-1/\alpha}\Big|^{p}}}
{\tau^{p/\beta}\epsilon^{p}} \ \ \text{ using \textbf{(UB)}.}
\end{split}
\end{equation}
For choice of $p$ such that  $\alpha <p<\alpha (K+1)$ in the Markov inequality in the third step of \eqref{psi-bound}, it is known
from  p.1451 of \cite{samorodnitsky:2004a} that
$$
\E \Big|C_{\alpha}^{1/\alpha}\displaystyle \sum\limits_{j=K+1}^{\infty}{\xi_{j}\Gamma_{j}^{-1/\alpha}\Big|^{p}}<\infty,
$$
we see that  \eqref{psi-bound} gives an integrable upper bound for $\psi_n (\epsilon,\delta,\tau^{1/\beta}) $ on $(1,\infty)$.
By a similar argument using the DCT, we have
\begin{equation*}
\label{T3-lim}
T_3^{(n)}(\epsilon)\to 0 \text{ as } n\to \infty.
\end{equation*}

Using \eqref{MoM-1}, we complete the proof by noting
\begin{equation}
\begin{aligned}
\limsup_{n\to \infty} \E\big[b_n^{-\beta}M_n^{\beta}\big] &\leq  \int_{0}^{\infty}\P \big(\Gamma_1^{-1/\alpha}
>C_\alpha^{-1/\alpha}\tau^{1/\beta}(1-\delta) \big)d\tau \nonumber\\
&= \int_{0}^{\infty} \left(1-\exp\big(-C_\alpha \tau^{-\alpha/\beta}(1-\delta)^{-\alpha}\big)\right)d\tau.
\end{aligned}
\end{equation}
By letting $\delta\to 0^{+}$, and applying DCT again and using \eqref{lim-maximal}, we
have
\begin{equation*}
\label{limsup}
\limsup_{n\to \infty} \E\big[n^{-d\beta/\alpha}M_n^{\beta} \big]\leq  \tilde{c}^{\beta}_{\bY}C_{\alpha}^{\beta/\alpha}
\E\big[Z_{\alpha/\beta}\big].
\end{equation*}

The argument for establishing a corresponding lower bound is similar. We  start with the following lower
bound for the tail distribution of $b_n^{-1} M_n$  from \cite{samorodnitsky:2004a},
\begin{equation}\label{bound-low}
\P \big(b_n^{-1} M_n> \lambda \big) \geq \P \big(C_{\alpha}^{1/\alpha}\Gamma_1^{-1/\alpha}>\lambda(1+\delta) \big)
-\phi_n (\epsilon,\lambda)-\widetilde{\psi}_n (\epsilon,\delta,\lambda),
\end{equation}
where $\phi_n (\epsilon,\lambda)$ is the same as in \eqref{phi} and $\widetilde{\psi}_n (\epsilon,\delta,\lambda)$
is defined by
\begin{equation}
\label{tilde-psi}
\begin{split}
\widetilde{\psi}_n (\epsilon,\delta,\lambda) &= \P\bigg(\max_{k\in [{\bf{0}},(n-1)\bf{1}] }\bigg\lvert \sum_{j=1}^{\infty}
\dfrac{\xi_j \Gamma_j^{-1/\alpha} |f_k(U_j^{(n)})|}{\max_{m\in [{\bf{0}},(n-1)\bf{1}]}|f_m(U_j^{(n)})|}\bigg\lvert >
\frac{ \lambda} {C_{\alpha}^{1/\alpha}\|f\|_{\alpha}}, \ \nonumber\\
& \qquad\qquad\qquad  \Gamma_1^{-1/\alpha}\leq \frac{b_n\lambda(1+\delta)}{C_{\alpha}^{1/\alpha}\|f\|_{\alpha}}, 
\text{ and } \Gamma_2^{-1/\alpha}\leq \frac{b_n \lambda \epsilon}{C_{\alpha}^{1/\alpha}\|f\|_{\alpha}} \bigg)\nonumber \\
&\leq  n^{d}\P\bigg(\Big\lvert \sum_{j=1}^{\infty}\xi_j\Gamma_j^{-1/\alpha}\Big\lvert >  \frac{b_n \lambda}
{C_{\alpha}^{1/\alpha}\|f\|_{\alpha}}, \
   \Gamma_1^{-1/\alpha}\leq \frac{b_n\lambda(1+\delta)}{C_{\alpha}^{1/\alpha}\|f\|_{\alpha}}, \nonumber \\
& \qquad \qquad \qquad \qquad \qquad\qquad\qquad \;\;\;\;\;  \text{ and } \Gamma_2^{-1/\alpha}\leq \frac{b_n \lambda \epsilon}
{C_{\alpha}^{1/\alpha}\|f\|_{\alpha}} \bigg).
\nonumber
\end{split}
\end{equation}
By a similar argument leading to (\ref{MoM-1}), we obtain
\begin{equation}\label{MoM-1b}
\begin{split}
\E\big[b_n^{-\beta} M_n^{\beta} \big]
&\geq   \int_0^\infty \P\big(C_{\alpha}^{1/\alpha}\Gamma_1^{-1/\alpha}>\tau^{1/\beta}(1+\delta)\big)d\tau  \\
&\qquad -  \int_0^\infty \phi_n (\epsilon,\tau^{1/\beta})d\tau -\int_0^\infty \widetilde{\psi}_n (\epsilon,\delta,\tau^{1/\beta})d\tau   \\
&:= \widetilde{T}_1(\delta) - T_2^{(n)}(\epsilon)  -  \widetilde{T}_3^{(n)}(\epsilon, \delta).
\end{split}
\end{equation}
By applying DCT with the integrable bounds derived in \eqref{phi-bound} and \eqref{psi-bound}, we derive
\begin{equation*}
\liminf_{n\to \infty} \E\big[n^{-d\beta/\alpha}M_n^{\beta}\big]\geq  \tilde{c}^{\beta}_{\bY} C_{\alpha}^{\beta/\alpha}
\E\left[Z_{\alpha/\beta}\right].
\end{equation*}
Combining the above inequalities, we prove \eqref{label-not conservative moments}, that is
$$n^{{-d\beta }/{\alpha}} \E\big[M_{n}^{\beta}\big] \to C \text{ as } n \to \infty.$$

In the case of a conservative action, let $\bW$ be a stationary $\sas$ random field
independent of $\bY$, having a similar integral representation with $\sas$ measure $M'$
on space $S'$ with control measure $\mu'$, independent of $M$ in the integral representation
of $\bY$. That is,
\begin{equation}
\begin{aligned}
W_t 
&= \int_{S'} c'_t(s){\left(\frac{d \mu^{'} \circ \phi^{'}_t}{d
\mu^{'}}(s)\right)}^{1/\alpha}g \circ \phi^{'}_t(s) M'(ds),\;\; t \in
\mathbb{Z}^d. \nonumber
\end{aligned}
\end{equation}
Denoting the above integrand by $g_{t}(s)$, further let $\bW$ be such that the sequence
$$b_{n}^{W}=\left(\int_{S'}\displaystyle\max\limits_{{\bf 0} \leq t \leq (n-1){\bf 1}}|g_{t}(s)|^{\alpha}\mu'(ds)\right)^{1/\alpha},
\;\;\; n\ge 1,
$$
satisfies equation (4.6) in \cite{samorodnitsky:2004a} (which is \textbf{(LB)} in the above) for some $\theta >0$ .\\
Define $ \bZ=\bW+\bY$. Then $\bZ$ inherits its natural integral representation on $S\cup S'$ and the
naturally defined action on that space is a stationary $\sas$ random field generated by a conservative
$\mathbb{Z}^{d}$-action. The deterministic maximal sequence $b_{n}^{Z}$ 
corresponding to conservative $\bZ$ satisfies (4.6) in \cite{samorodnitsky:2004a}  as
$$b_{n}^{Z}\geq b_{n}^{W} \text{ for all } n.$$
Using symmetry, we have
\begin{equation}
\label{label-relation between two processes}
\P\left(M_{n}^{Z}>x\right)\geq \frac{1}{2}\P\left(M_{n}>x\right)
\end{equation}
and
\begin{align}
\E\big[n^{-d\beta/\alpha} M_n^{\beta}\big]
&= \int_{0}^{\infty}\P\left(n^{{-d}/{\alpha}}{M_{n}}>\tau^{1/\beta}\right)d\tau  \nonumber\\
&\leq  2\int_{0}^{1}\P\left((b_n^{Z})^{-1}{M_{n}^Z}>C\tau^{1/\beta}\right)d\tau \nonumber\\
&\qquad + 2\int_{1}^{\infty}\P\left((b_n^{Z})^{-1}{M_{n}^Z}>C\tau^{1/\beta}\right)d\tau \nonumber\\
&= S_n^{(1)} + S_n^{(2)} \nonumber
\end{align}
with the second step following from \eqref{label-relation between two processes} and that $n^{{-d}/{\alpha}}b_n^{Z}$
converges to $0$ and hence is bounded by a constant $1/C$ say.
We use the fact from \cite{samorodnitsky:2004a} that
$$n^{-d/\alpha}M_n \to 0 \text{ as } n \to \infty,$$
under the conditions (4.6) and (4.8) in the afore-mentioned reference
and conclude \eqref{label-conservative moments} via a DCT argument by using the trivial bound
$$\P\left((b_n^{Z})^{-1}{M_{n}^Z}>C\tau^{1/\beta}\right) \le 1$$
for $\tau \in (0,1)$ and obtaining a non-trivial integrable
bound for the same on $(1,\infty)$.
Again with a similar choice of $\epsilon$ as in the dissipative case
we have
\begin{align*}
 \P\big(M_{n}^{Z}>C b_{n}^{Z}\tau^{1/\beta} \big)  &\leq  \P\big(\Gamma_{1}^{-1/{\alpha}}> C\tau^{1/\beta}\epsilon \big)   \\
& \; \; + \P\big(M_{n}^{Z}>C b_{n}^{Z} \tau^{1/\beta}, \;\Gamma_{1}^{-1/{\alpha}}\leq C \tau^{1/\beta}\epsilon\big),
\end{align*}
where $M_{n}^{Z}$ is the maxima, and $b_{n}^{Z}$ is the corresponding deterministic maximal sequence for $\bZ$.
Let $\bZ$ have a series representation in terms of arrival times of a unit Poisson process, $\Gamma_j$ and
Rademacher variables $\xi_j$. Now choose $K$ large enough so that $\alpha(K+1) > d/\theta$. For $p$ satisfying
$$
\frac{d}{\theta} <p<\alpha (K+1),$$
using a technique similar to \eqref{psi-bound} by an application of Markov's inequality,
we derive an integrable upper bound for $\tau \in (1,\infty)$ as
\begin{equation}
\label{bound-tight}
\begin{split}
& \P\left(M_{n}^{Z}>C b_{n}^{Z}\tau^{1/\beta} , \;\;\Gamma_{1}^{-1/{\alpha}}\leq C\tau^{1/\beta} \epsilon\right)   \\
&\leq  n^{d}\P\bigg(\Big\lvert C_{\alpha}^{1/\alpha}\sum_{j=1}^{\infty}\xi_j\Gamma_j^{-1/\alpha}\Big\lvert > \frac{C b_n^{Z} \tau^{1/\beta}}
{\|f^{Z}\|_{\alpha}}, \  C_{\alpha}^{1/\alpha}\Gamma_1^{-1/\alpha}\leq \frac{C b_n^{Z}\tau^{1/\beta}\epsilon}{\|f^{Z}\|_{\alpha}} \bigg)\\
& \leq \,  n^{d}\P\bigg(\Big\lvert C_{\alpha}^{1/\alpha}\sum_{j=1}^{\infty}\xi_j\Gamma_j^{-1/\alpha}\Big\lvert > \frac{C b_n^{Z} \tau^{1/\beta}}
{\|f^{Z}\|_{\alpha}},  \\
& \qquad \qquad \qquad \qquad\;\;\;\;\;\;C_{\alpha}^{1/\alpha}\Gamma_j^{-1/\alpha}\leq \frac{C b_n^{Z}\tau^{1/\beta}\epsilon}{\|f^{Z}\|_{\alpha}}
 \text{ for all } j\in \mathbb{N}\bigg)\\
&\leq  n^{d}\P\bigg(C_{\alpha}^{1/\alpha}\Big|\sum_{j=K+1}^{\infty}\xi_j\Gamma_j^{-1/\alpha}\Big|>C\|f^{Z}\|_{\alpha}^{-1}
b_n^{Z}\tau^{1/\beta}(1-K\epsilon)\bigg) \\
&\leq  n^{d} (b_{n}^{Z})^{-p}C^{p}\dfrac{\|f^{Z}\|_{\alpha}^{p}\E \Big|C_{\alpha}^{1/\alpha}\displaystyle
\sum\limits_{j=K+1}^{\infty}{\epsilon_{j}\Gamma_{j}^{-1/\alpha} \Big|^{p}}}{\tau^{p/\beta}\epsilon^{p}}   \\
&\leq A C^{p} \|f^{Z}\|_{\alpha}^{p} \cfrac{\E \Big|C_{\alpha}^{1/\alpha}\displaystyle
\sum\limits_{j=K+1}^{\infty}{\epsilon_{j}\Gamma_{j}^{-1/\alpha}\Big|^{p}}}{\tau^{p/\beta}\epsilon^{p}}.
\end{split}
\end{equation}

Observing that
$$
\int_{0}^{\infty}{\P\big(\Gamma_{1}^{-1/{\alpha}}> \epsilon\tau^{1/\beta}\big)}d\tau = \epsilon^{-\beta} \E\big[\bZ_{\alpha/\beta}\big]
= \epsilon^{-\beta}{{\Gamma(1-{\beta}/{\alpha})}}<\infty,
$$
and using integrable bound for
$$\P\big(M_{n}^Z>\tau^{1/\beta} b_{n}^Z, \;\;\Gamma_{1}^{-1/{\alpha}}\leq \epsilon\tau^{1/\beta}\big)
$$
as derived in \eqref{bound-tight}, we obtain a nontrivial bound for $S_n^{(2)}$.  Equipped to apply DCT
with the trivial bound $1$ for $S_n^{(1)}$ and an integrable bound for $S_n^{(2)}$, we conclude \eqref{label-conservative moments}. This completes the proof. 

\subsection{Proof of Theorem~\ref{weak:eff:dim}}

The proof again follows by noting that
\begin{equation*}
\begin{split} \label{Eq:md}
\E\big[b_n^{-\beta} M_n^{\beta}\big]
&= \int_{0}^{\infty}\P\big(b_n^{-1} M_n>\tau^{1/\beta}\big)d\tau   \\
&\leq \int_{0}^{\infty}\Big\{\P\big(\Gamma_{1}^{-1/{\alpha}}> \tau^{1/\beta}\epsilon \big)   \\
&\qquad \qquad +\P\big(M_{n}>\tau^{1/\beta} b_{n}, \;\;\Gamma_{1}^{-1/{\alpha}}\leq \tau^{1/\beta}\epsilon\big)\Big\}d\tau   \\
&= \epsilon^{-\beta}{{\Gamma(1-{\beta}/{\alpha})}} +\int_{0}^{1}\P\big(M_{n}>\tau^{1/\beta} b_{n}, \;\;
\Gamma_{1}^{-1/{\alpha}}\leq \tau^{1/\beta}\epsilon\big)d\tau  \\
&\qquad +\int_{1}^{\infty}\P\big(M_{n}>\tau^{1/\beta} b_{n}, \;\;\Gamma_{1}^{-1/{\alpha}}\leq \tau^{1/\beta}\epsilon\big)d\tau.  \\
\end{split}
\end{equation*}
The  integral over $[0, 1]$ is bounded by 1. To bound the integral over $(1, \infty)$,  we choose $K$ large enough so that
$\alpha (K+1) > \frac{\alpha d}{\theta_1}$. Fix $\epsilon$ satisfying $0< \epsilon < \frac{1}{K}$ and $p$ satisfying
$$\frac{\alpha d}{\theta_1}<p<\alpha (K+1).
$$
The same argument as in \eqref{bound-tight}, together with  the lower bound in \eqref{weak:eff}, gives
\[
\P\big(M_{n}>\tau^{1/\beta} b_{n}, \;\;\Gamma_{1}^{-1/{\alpha}}\leq \tau^{1/\beta}\epsilon\big) \le \frac{B}{{\tau^{p/\beta}\epsilon^{p}}},
\]
where
$$B=A \|f\|_{\alpha}^{p} {\E\bigg|C_{\alpha}^{1/\alpha}\displaystyle \sum\limits_{j=K+1}^{\infty}{\epsilon_{j}
\Gamma_{j}^{-1/\alpha}\bigg|^{p}}}.
$$
It follows from above that
\begin{equation*}
\begin{split}
\E\big[b_n^{-\beta} M_n^{\beta}\big] &\leq  \epsilon^{-\beta}{{\Gamma(1-{\beta}/{\alpha})}} +1+
 \int_{1}^{\infty} \frac{B}{\tau^{p/\beta}\epsilon^{p}} d\tau \nonumber \\
&= K_1< \infty.
\end{split}
\end{equation*}
Hence
$$\E\big[ M_n^{\beta}\big]\leq K_1 \cdot b_n^{\beta} \leq K_1c_2 \cdot n^{\beta\theta_2/\alpha} $$
 for all sufficiently large $n$, say $n \ge n_0$. Taking  $K^\prime = \max\{c_2  K_1 ;\E\big[ M_k^{\beta}\big], k \le n_0\}$
 yields \eqref{weak:eff:bound}.

\subsection{Proof of Theorem~\ref{thm3:maximal_moment}}

(1)~When the action $\{\phi_t\}_{t\in F}$ restricted to the free group $F$ is dissipative, then by Proposition~5.1 of \cite{roy:samorodnitsky:2008}, the
sequence $\{b_n\}_{n\geq 0}$ satisfies
$$\lim_{n\to \infty} n^{-p/\alpha} b_n=c, \text{ a constant},$$
which implies that $b_n$ satisfies (4.6) in \cite{samorodnitsky:2004a} with $\theta=p/\alpha$. Also, (4.17) of \cite{roy:samorodnitsky:2008} holds; see the proof of Theorem~5.4 in \cite{roy:samorodnitsky:2008}.

Now we choose $K$ such that $\alpha(K+1) > d \alpha/p$, use the same tail bound as in \eqref{bound-ini} and apply DCT using integrable bounds on
$$\phi_n(\epsilon, \tau^{1/\beta})\leq n^{d}b_n^{-\alpha}\cfrac{C_\alpha^{-1}\epsilon^{-\alpha}}{\tau^{\alpha/\beta}}
\leq K_2{C_\alpha^{-1}\epsilon^{-\alpha}}{\tau^{-\alpha/\beta}},
$$
\[
\begin{split}
\psi_n(\epsilon,\delta,\tau)&\leq n^{d}b_n^{-p'}\dfrac{\|f\|_{\alpha}^{p'}\E \bigg|C_{\alpha}^{1/\alpha}\displaystyle
 \sum\limits_{j=K+1}^{\infty}{\epsilon_{j}\Gamma_{j}^{-1/\alpha}\bigg|^{p'}}}{\tau^{p'/\beta}\epsilon^{p'}} \\
&\leq K_3\epsilon^{-p'} \tau^{-p'/\beta},
\end{split}
\]
for $p'$ satisfying $$\frac{d\alpha}{p}\leq p' \leq \alpha (K+1).$$
Then as in the proof of \eqref{label-not conservative moments},
\eqref{label-G, not conservative moments} follows.

\noindent{(2)}~When the action $\left\{\phi_{t}\right\}_{t\in F}$ is conservative, we can obtain a stationary $\sas$ random field $\bZ$
generated by a conservative $\mathbb{Z}^{d}$-action such that $b_{n}^{Z}$ satisfies (4.6) in \cite{samorodnitsky:2004a} for some
$\theta > 0$ and $$n^{-p/\alpha}b_n^{Z} \to 0 \text{ as } n\to \infty.$$ Again by the exact argument used to prove
\eqref{label-conservative moments}, we obtain \eqref{label-G, conservative moments}.





\appendix
\section{Maximal moments for continuous parameter case}
\label{Appendix}
Here we present the theorem on the rate of growth of moments of maximum of $S\alpha S$
 process indexed by continuous time in $\bbR$, which can be easily extended to the class of fields indexed by $\bbR^{d}$ (see Remark~\ref{remark:last} below).

\begin{thm}
\label{cont:parameter:gom}
Let $\bY=\{Y(t)\}_{t\in \bbR}$ be a stationary measurable $S\alpha S$ process with $0<\alpha< 2$ and
having integral representation as
\begin{equation*}
Y(t) \eqdef \int_S f_t(s)M(ds)=\int_S c_t(s){\left(\frac{d \mu \circ \phi_t}{d
\mu}(s)\right)}^{1/\alpha}f \circ \phi_t(s) M(ds),\;\; t \in
\bbR, \label{repn_integral_stationary_cont}
\end{equation*}
where $f \in L^\alpha(S, \mu)$, $\{\phi_t\}_{t \in \mathbb{R}}$ is a nonsingular flow,  $\{c_t\}_{t \in \mathbb{R}}$ is a $\pm$-valued cocycle with respect to  $\{\phi_t\}_{t \in \mathbb{R}}$ and $M$ is an $\SaS$ measure with control measure $\mu$; see \cite{rosinski:1995}.
\begin{enumerate}
\item  If $\bY$ is generated by a dissipative flow, that is $\bY$ has a mixed moving average representation given by
\begin{equation*}
\label{MMR2}
\bY\stackrel{d}{=}\left\{\int_{W\times\mathbb{Z}}f(v,t+s)M(dv,ds)\right\}_{t\in\bbR},
\end{equation*}
then, for $0<\beta<\alpha$,
\begin{equation}
\label{label-not conservative moments-cont}
\E\big[T^{-\beta/\alpha }M_{T}^{\beta}\big] \to C \text{ as } \, T \to \infty,
\end{equation}
where constant $C=c^{\beta}C_{\alpha}^{\beta/\alpha}\E\big[\bZ_{\alpha/\beta}\big]$, with $\bZ$
denoting a Frechet random variable with shape parameter $\alpha/\beta$ and constant
$$c=\lim_{T\to\infty} T^{-1/\alpha}b_T, \text{ and}$$
 $C_\alpha$ is as defined in \eqref{C-alpha}.
 \item If $\bY$ is generated by a conservative flow, then
\begin{equation}
\label{label-conservative moments-cont}
\E\big[T^{{-\beta }/{\alpha}}M_{T}^{\beta}\big] \to 0\ \ \text{ as } \, T \to \infty. \\
\end{equation}
\end{enumerate}
\label{label-Continuous MoM}
\end{thm}

\begin{proof} Stationarity and measurability together implies continuity in probability for stable processes (see Proposition~3.1 of \cite{roy:2010b}). Therefore following \cite{samorodnitsky:2004b}, we shall approximate the stable process (and all of its functionals) by its dyadic skeletons even without writing it explicitly at times. This will ensure, in particular, that every quantity considered in this proof is measurable.

As in the proof of Theorem 3.1, we consider cases (1) and (2) separately. When $\bY$ is generated by a dissipative flow, the deterministic family
 $$\{b_T\}_{T\geq 0}=\left\{\left(\int_{S}\sup_{0\leq t \leq T}|f_t(s)|^{\alpha} \mu(ds)\right)^{1/\alpha}\right\}_{T\geq 0}$$
 satisfies
$ \lim_{T\to\infty} T^{-1/\alpha}b_T=c, \text{ a constant}.$ The above implies that $b_T$ satisfies conditions (2.9) with $\theta=1/\alpha$ and (2.12)  in \cite{samorodnitsky:2004b}, analogous to \textbf{(LB)} and \textbf{(LL)} in Theorem \ref{label-Discrete MoM} for fields indexed by $\mathbb{Z}^d$.
For a choices of $\epsilon>0$ and $0<\delta<1$ such that $\epsilon$ is chosen small enough as compared to $\delta$
 and for $K=0,1,2, ...$ satisfying
$$K< \cfrac{1}{\epsilon C_\alpha^{1/\alpha}},$$
we bound the tail distribution of $b_T^{-1} M_T$  as
\begin{equation}
\label{bound-ini-cont}
\P \big(b_T^{-1} M_T> \lambda \big)\leq \P\big(C_{\alpha}^{1/\alpha}\Gamma_1^{-1/\alpha}>\lambda(1-\delta)\big)
+\phi_T (\epsilon,\lambda)+\psi_T (\epsilon,\delta,\lambda),
\end{equation}
taken from \cite{samorodnitsky:2004b}.
The quantities $\phi_T$ and $\psi_T$ in \eqref{bound-ini-cont} are defined and bounded as follows:
\begin{equation}
\label{phi-cont}
\begin{split}
\phi_T (\epsilon,\lambda) &= \P\bigg( \text{for some } t\in [0,T], \ \dfrac{ \Gamma_{j}^{-1/\alpha} |f_t(U_j^{(T)})|}{\sup_{s\in [0,T]}|f_s(U_j^{(T)})|}>\epsilon \lambda  \\
&\qquad \;\;\qquad   \qquad \qquad \text{ for at least 2 different } j \bigg)  \\
&\leq  \lfloor T \rfloor \P\bigg( \;\Gamma_{j}^{-1/\alpha}\sup_{0\leq t \leq 1}\dfrac{|f_t(U_j^{(T)})|}{\sup_{s\in [0,T]}|f_s(U_j^{(T)})|}>\epsilon \lambda \\
&\qquad \qquad \qquad \qquad \;\; \text{ for at least 2 different } j\bigg),
\end{split}
\end{equation}
where $\lfloor T \rfloor$ denotes the smallest integer $\ge T$ and the inequality follows from the same
argument as in (2.26) of \cite{samorodnitsky:2004b}. Furthermore, the random points
$$b_T\Gamma_{j}^{-1/\alpha}\dfrac{|f_t(U_j^{(T)})|}{\sup_{s\in [0,T]}|f_s(U_j^{(T)})|}, \quad j = 1, 2, \ldots$$
have the same distribution as
$$Z_j(t)=b_1\Gamma_{j}^{-1/\alpha}\dfrac{|f_t(V_j)|}{ \sup_{s\in [0,1]} |f_s(V_j)|}, \quad j = 1, 2, \ldots,$$
where $\{V_j\}$ is identically distributed as $\{U_j^{(1)}\}$ and independent of $\{\Gamma_j\}$. This
and \eqref{phi-cont} imply that
\begin{equation*}
\label{phi-cont2}
\begin{split}
\phi_T (\epsilon,\lambda)  &\le \lfloor T \rfloor \P\bigg( \;b_1\Gamma_{j}^{-1/\alpha}\sup_{0\leq t \leq 1}
\dfrac{|f_t(V_j)|}{ \sup_{s\in [0,1]} |f_s(V_j)|}>b_T\epsilon \lambda\nonumber \\
&\qquad \qquad \qquad \qquad  \;\; \text{ for at least 2 different } j \bigg) \nonumber \\
&=  \lfloor T \rfloor \P\left(\sum_{j=1}^{\infty} 1_{\{\sup_{t\in [0,1]} |Z_j(t)|\}}(b_T\epsilon \lambda,\infty)\geq 2\right).
\end{split}
\end{equation*}
For set of interest
$$
B(T)=\bigg\{(z(t); t\in [0,1]): \sup_{t\in [0,1]}|z(t)|>b_T\epsilon \lambda\bigg\},
$$
$\{Z_j(t), j \ge 1\}$ 
 are points of a Poisson random measure with mean measure
 $$\Lambda(B(T))=\left(\cfrac{b_T\epsilon\lambda}{b_1}\right)^{-\alpha}.$$
Using the fact that $[T]b_T^{-\alpha}\leq K_4 \text{ a constant},$
we have
\begin{equation}
\label{phi-cont1}
\phi_T (\epsilon,\lambda)\leq K_4 b_1^{\alpha}\epsilon^{-\alpha}\lambda^{-\alpha}.
\end{equation}

Similarly as above,  we have for $0<\epsilon<\delta/K$,
\begin{equation}
\label{psi-cont}
\begin{split}
\psi_T (\epsilon,\delta,\lambda) &= \P\bigg(b_T\sup_{t\in [0,T]}\bigg\lvert \sum_{j=1}^{\infty}\xi_j
\Gamma_j^{-1/\alpha}  \cdot \dfrac{|f_t(U_j^{(t)})|}{\sup_{s\in [0,T]}|f_s(U_j^{(s)})|}\bigg\lvert
  > C_{\alpha}^{-1/\alpha}b_T \lambda; \\
  &\qquad \quad \quad  \qquad \;\; b_1\Gamma_1^{-1/\alpha}\leq b_T\lambda(1-\delta) \,\text{ and } b_1\Gamma_2^{-1/\alpha}\leq b_T\lambda\epsilon  \bigg)\nonumber \\
&\leq  \lfloor T \rfloor \P\bigg(b_1\sup_{t\in [0,1]}\bigg\lvert \sum_{j=1}^{\infty} \xi_j
\Gamma_j^{-1/\alpha}  \cdot  \dfrac{ |f_t(V_j)|}{\sup_{s\in [0,1]}|f_s(V_j)|}\bigg\lvert   > C_{\alpha}^{-1/\alpha}b_T \lambda; \\
&  \qquad \quad \quad  \qquad \;\; b_1\Gamma_1^{-1/\alpha}\leq b_T\lambda(1-\delta) \,
 \text{ and } b_1\Gamma_2^{-1/\alpha}\leq b_T\lambda\epsilon  \bigg).\nonumber \\
\end{split}
\end{equation}
Using the same argument as in (2.29) - (2.33) of \cite{samorodnitsky:2004b}, leveraging on the observation
that
$$b_T\Gamma_{j}^{-1/\alpha}\dfrac{|f_t(U_j^{(T)})|}{\sup_{s\in [0,T]}|f_s(U_j^{(T)})|}$$
are identically distributed as $Z_j(t)$ and applying an exponential Markov inequality in the penultimate step,
we derive
\begin{equation}
\label{psi-cont2}
\begin{split}
\psi_T (\epsilon,\delta,\lambda) &\leq \lfloor T \rfloor\P\bigg(\sup_{t\in [0,1]}\bigg\lvert \sum_{j=K+1}^{\infty}\xi_j \Gamma_j^{-1/\alpha}\cdot\dfrac{|f_t(V_j)|}{\sup_{s\in [0,1]}|f_s(V_j)|}\bigg\lvert \nonumber\\
&\qquad \qquad \qquad \qquad \;\;\;\; \;\;\;\;>b_T\big(1-\epsilon C_\alpha^{1/\alpha}\big)b_1^{-1} C_\alpha^{-1/\alpha}\lambda\bigg)
\nonumber \\
&\leq 4 \lfloor T \rfloor \int_{0}^{\infty}\exp(-x) \frac{x^K}{K!}\exp\bigg\{-\cfrac{\big(1-\epsilon C_\alpha^{1/\alpha}\big) \lambda \log 2}
{(\gamma+2x^{-1/\alpha}b_T)b_1C_\alpha^{1/\alpha}}  \bigg\} dx\nonumber \\
&\leq  4\lfloor T \rfloor\bigg(C_1 \exp(-\zeta(\lambda)T^{\theta})  +  \int_{0}^{1}  \frac{x^K}{K!} \exp\big(-x -C_2\lambda x^{1/\alpha}T^{\theta}
\big)dx\bigg),
\end{split}
\end{equation}
where $\zeta(\lambda)$ is an increasing function of $\lambda$.

For any $0<\beta<\alpha$, using the tail bound in \eqref{bound-ini-cont} we have
\begin{equation}
\label{MoM-1-cont}
\begin{split}
\E\big[b_T^{-\beta} M_T^{\beta}\big] &= \int_0^\infty \P \big(b_T^{-1} M_T> \tau^{1/\beta}\big)d\tau \nonumber \\
&\leq   \int_0^\infty \P \big(C_{\alpha}^{1/\alpha}\Gamma_1^{-1/\alpha}>\tau^{1/\beta}(1-\delta)\big)d\tau \nonumber \\
&\qquad + \int_0^\infty \phi_T (\epsilon,\tau^{1/\beta})d\tau +\int_0^\infty \psi_T (\epsilon,\delta,\tau^{1/\beta})d\tau  \nonumber \\
&= T_1(\delta) + T_2^{(T)}(\epsilon)  + T_3^{(T)}(\epsilon, \delta).
\end{split}
\end{equation}
Again from \cite{samorodnitsky:2004b}, we know that, as $T \to \infty$, $\phi_T (\epsilon,\tau)$ and
$\psi_T (\epsilon,\delta,\tau^{1/\beta})$ converge to $0$  point-wise for all $\tau\in [0,\infty).$ Hence
using the integrable bounds derived in \eqref{phi-cont} and \eqref{psi-cont} on $(1,\infty)$ and the
trivial bound $1$ on $(0,1)$, we apply DCT to conclude that $$T_2^{(T)}(\epsilon),\;T_3^{(T)}(\epsilon, \delta)
\to 0 \text{ as } T\to \infty,$$
which gives
\begin{equation}
\begin{aligned}
\limsup_{T\to \infty} \E\big[b_T^{-\beta}M_T^{\beta}\big]&\leq  \int_{0}^{\infty}\P(\Gamma_1^{-1/\alpha}>C_\alpha^{-1/\alpha}\tau^{1/\beta}(1-\delta))d\tau \nonumber\\
&= \int_{0}^{\infty} \left(1-\exp\big(-C_\alpha \tau^{-\alpha/\beta}(1-\delta)^{-\alpha}\big)\right)d\tau.
\end{aligned}
\end{equation}
By letting $\delta\to 0^{+}$, and applying DCT again gives
$$\limsup_{T\to \infty}  \E \big[T^{-\beta/\alpha}M_T^{\beta}\big] \leq  c^{\beta}C_{\alpha}^{\beta/\alpha}\E\left[\bZ_{\alpha/\beta}\right].$$
On the other hand, we can use a similar a lower tail bound
\begin{equation*}
\E\big[b_T^{-\beta} M_T^{\beta}\big] \geq T_1(\delta) - T_2^{(T)}(\epsilon)  - T_3^{(T)}(\epsilon, \delta).
\end{equation*}
and applying DCT with the integrable bounds derived in \eqref{phi-cont} and \eqref{psi-cont}, we have
\begin{equation*}
\label{liminf}
\liminf_{T\to \infty} \E\big[T^{-\beta/\alpha}M_T^{\beta}\big]\geq  c^{\beta}C_{\alpha}^{\beta/\alpha} \E\left[\bZ_{\alpha/\beta}\right].
\end{equation*}
This concludes the proof of \eqref{label-not conservative moments-cont}.

\noindent{(2).} Consider a stationary $S\alpha S$ random field $\bW$ independent of $\bY$, also given by the integral
representation of the form
$$ \bW= \int_{S'}g_{t}(s)M'(ds),\;\;\;t\in \mathbb{R},$$
where $M'$ is a $S\alpha S$ random measure with control measure $\mu'$, independent of $M$ in the integral
representation of $\bY$
and generated by a conservative flow and also satisfying
\begin{equation*}
\label{cond-1-cont}
b^W_T\geq cT^{\theta} \text{ for sufficiently large } T
\end{equation*}
for some $\theta >0$.
Define $\bZ=\bY+\bW$, a stationary $S\alpha S$ random process generated by a conservative $\mathbb{R}$-action with the natural integral representation on $S\cup S'$ corresponding to the naturally defined action on that space. Let ${b}^{Z}_T$ be the corresponding deterministic maximal quantity defined for the process $Z$. As ${b}_{T}^Z\geq {b}_{T}^Y$ for all $T>0$, the conservative process $Z$ satisfies \eqref{cond-1-cont}.
\begin{equation*}
\begin{split}
\E\big[T^{-\beta/\alpha} M_T^{\beta}\big]
&= \int_{0}^{\infty} \P\big(T^{{-1}/{\alpha}}{M_{T}}>\tau^{1/\beta}\big)d\tau  \nonumber\\
&\leq  2\int_{0}^{1} \P\big((b_T^{Z})^{-1}{M_{T}^Z}>C\tau^{1/\beta}\big)d\tau \nonumber\\
&\qquad \qquad + 2\int_{1}^{\infty}\P\big((b_T^{Z})^{-1}{M_{T}^Z}>C\tau^{1/\beta}\big)d\tau \nonumber\\
&= S_T^{(1)} + S_T^{(2)}
\end{split}
\end{equation*}
with the second step following from symmetry and the fact $T^{{-1}/{\alpha}}b_T^{Z}$ is bounded by $C^{-1}$,
a constant. Using the bounding technique in \eqref{psi-cont}, we have a similar integrable bound for
$$
\P\big(M_{T}^{Z}>\tau^{1/\beta} b_{T}^{Z}, \;\;\Gamma_{1}^{-1/{\alpha}}\leq\tau^{1/\beta} \epsilon\big),
$$
which leads to \eqref{label-conservative moments-cont} by a similar DCT argument using the fact from Theorem 2.2 of \cite{samorodnitsky:2004b}  that $\P\big(T^{-1/\alpha}M_T>\tau^{1/\beta}\big) \to 0  \text{ as } T\to \infty.$
\end{proof}

\begin{remark}\label{remark:last}
  \textnormal{The results presented in this section can easily be extended to stationary measurable symmetric
  $\alpha$-stable random fields indexed by $\mathbb{R}^d$. For simplicity of presentation, we only dealt with the $d=1$ case here. This extension to higher dimension can be done using the techniques of \cite{roy:2010b} and \cite{chakrabarty:roy:2013}. More specifically, the idea is to approximate the continuous parameter random field $\{X_t\}_{t \in \mathbb{R}^d}$ by its discrete parameter skeletons $\{X_t\}_{t \in  2^{-i}\mathbb{Z}^d}$, $i=0,1,2, \ldots$.}

  \textnormal{In \cite{chakrabarty:roy:2013}, the notion of effective dimension was extended to the continuous parameter case based on the following observation: the effective dimensions of $\{X_t\}_{t \in  2^{-i}\mathbb{Z}^d}$, $i=0,1,2, \ldots$ are equal and hence can be defined as the \emph{group theoretic dimension} of $\{X_t\}_{t \in \mathbb{R}^d}$. With this definition, Theorem~\ref{cont:parameter:gom} can be extended to the higher-dimensional case connecting the rate of growth of maximal moments to the group theoretic dimension $p$. We can also define a continuous parameter analogue of weak effective dimension and relate it to the asymptotic properties of the maximal moments. In summary, all the results presented in Section~\ref{sec:discrete} above can be rewritten for stationary measurable $\sas$ random fields indexed by $\mathbb{R}^d$.}
\end{remark}
\end{document}